\documentclass[ejsv2,preprint,floatsintext,noshowframe]{imsart}

\usepackage{amsfonts}
\usepackage{amsmath,bm}
\usepackage{mathrsfs,nicefrac,graphicx}
\usepackage{amsthm,nicefrac,bigints}
\usepackage[utf8]{inputenc}
\usepackage[T1]{fontenc}
\usepackage{epsfig}
\usepackage{bigints}
\usepackage{multirow}
\usepackage{eurosym}
\usepackage{enumerate}
\usepackage{enumitem}
\usepackage{bbm,bm}
\usepackage{xcolor}

\usepackage{array}
\newcolumntype{L}[1]{>{\raggedright\let\newline\\\arraybackslash\hspace{0pt}}m{#1}}
\newcolumntype{C}[1]{>{\centering\let\newline\\\arraybackslash\hspace{0pt}}m{#1}}
\newcolumntype{R}[1]{>{\raggedleft\let\newline\\\arraybackslash\hspace{0pt}}m{#1}}

\usepackage{float}
\usepackage{url}
\usepackage{hyperref}
\usepackage{marginnote}

\newtheorem{theorem}{Theorem} 
\newtheorem{proposition}[theorem]{Proposition}
\newtheorem{lemma}[theorem]{Lemma}
\newtheorem*{theorem*}{Theorem}
\newtheorem{corollary}[theorem]{Corollary}
\theoremstyle{definition}\newtheorem{remark}[theorem]{Remark}
\theoremstyle{definition}\newtheorem{example}[theorem]{Example}
\theoremstyle{definition}

\newcommand{\bE}{E}

\newcommand{\bN}{\mathbb{N}}
\newcommand{\bP}{\mathbb{P}}

\newcommand{\bR}{\mathbb{R}}

\newcommand{\bZ}{\mathbb{Z}}
\newcommand{\bH}{\mathbb{H}}

\newcommand{\bS}{\mathbb{S}}

\newcommand{\cB}{\mathcal{B}}
\newcommand{\cC}{\mathcal{C}}

\newcommand{\cF}{\mathcal{F}}

\newcommand{\cL}{\mathcal{L}}

\newcommand{\cP}{\mathcal{P}}

\newcommand{\cS}{\mathcal{S}}

\newcommand{\unif}{\text{\rm unif}}
\newcommand{\inv}{\text{inv}}
\newcommand{\rev}{\text{\rm rev}}

\newcommand{\cov}{\text{\rm cov}}

\newcommand{\todistr}{\to_{\text{\rm\tiny d}}}

\newcommand{\toweak}{\to_{\text{\rm\tiny w}}}
\newcommand{\eqdistr}{=_{\text{\rm\tiny d}}}

\newcommand{\ind}{_{\text{\rm\tiny ind}}}

\newcommand{\pit}{\scalebox{0.65}{\hspace*{0.5mm}\rm P}}
\newcommand{\gr}{\scalebox{0.65}{\hspace*{0.5mm}\rm G}}

\newcommand{\ba}{\scalebox{0.65}{\hspace*{0.5mm}\rm B}}
\newcommand{\bam}{\scalebox{0.65}{\hspace*{0.5mm}\rm B,-}}
\newcommand{\bap}{\scalebox{0.65}{\hspace*{0.5mm}\rm B,+}}

\newcommand{\brend}{\hfill $\triangleleft$} 

\allowdisplaybreaks

\newcommand{\DE}{D\hspace*{-0.4mm}E}

\newcommand{\nf}{\nicefrac}

\begin{document}

\begin{frontmatter}

\title{Efficiency of pattern-based independence test}
\runtitle{Pattern-based independence tests}

\begin{aug}
\author{\fnms{Ludwig} \snm{Baringhaus}\ead[label=e1]{lbaring@stochastik.uni-hannover.de}}
\and
\author{\fnms{Rudolf} \snm{Gr{\"u}bel}\corref{}
        \ead[label=e2]{rgrubel@stochastik.uni-hannover.de}}
        \address[]{Institute of Actuarial and Financial Mathematics\\
           Leibniz Universit\"at Hannover\\ Welfengarten 1\\
           D-30167 Hannover, Germany \printead{e1,e2}}
\runauthor{L. Baringhaus and R. Gr{\"u}bel}
\end{aug}

\begin{abstract}
Tests of independence are an important tool in applications, 
specifically in connection with the detection of a relationship between variables; they also 
have initiated many developments in statistical theory. In the present paper we build upon
and extend a recently established link to Discrete Mathematics and Theoretical Computer 
Science, exemplified by the appearance of copulas in connection with limits of permutation 
sequences, and by the connection between quasi-randomness and consistency of pattern-based 
tests of independence. The latter include classical procedures, such as  Kendall's tau,
which uses patterns of length two.  Longer patterns lead to tests that are consistent against 
large classes of alternatives, as first shown by Hoeffding (1948) with patterns of length five, 
and by  Yanagimoto (1970) and Bergsma and Dassios (2014) for patterns of length four. 
More recently Chan et al.\ (2020) characterized quasi-randomness for sets of patterns of length four,
which leads to several new consistent pattern-based test for independence. 
We give a detailed and complete description of the respective limiting null distributions. 
In connection with the power performance of the tests, which is of interest for practical purposes,
we provide results on their (local) asymptotic relative efficiencies. We also include a small 
simulation study that supports our theoretical findings. 
\end{abstract}



\end{frontmatter}

\date{\today}

\section{Introduction}\label{sec:intro}
Tests of independence are a major topic in applications; 
the chi-squared test and Kendall's tau, for example,  are well-known across the sciences. Such tests 
provide an important tool in connection with detecting dependencies between variables from noisy data. 
Also, throughout their long history, the construction and analysis of such procedures have gone
hand-in-hand with theoretical advances, such as  Hoeffding's decomposition for the
derivation of limiting null distributions, or the use of large deviation principles for the behavior 
under alternatives.  

In recent years an interesting connection to Discrete Mathematics and Theoretical Computer Science has opened up: First, two-dimensional copulas, often used for modelling dependencies of two real random variables, appear  as limits of sequences of permutations, where they are known as \emph{permutons}; 
see~\cite{GrMet} for a simple introduction. Secondly, the
associated notion of convergence refers to pattern frequencies, which are the basis for several 
test statistics, including Kendall's tau. The important statistical notion of consistency, i.e.\ the property that 
non-independence will be discovered in the limit, is then closely related to the concept of
\emph{quasi-randomness}  in these other areas. (Indeed, a quasi-randomness property has been 
found in Statistics about 40 years \emph{avant la lettre}; see \cite[p.\,3]{CDN}.) In the present 
paper we build on and extend this connection:  The class of all sums of sets of patterns of length 
four that imply quasi-randomness has been characterized in~\cite{Chan}, and we investigate 
the associated consistent tests of independence. One of these pattern sums was already 
discovered in~\cite{BergsmaDassios} in connection with tests of independence, and in this case,
the limiting null distribution has been obtained  in~\cite{NWD}.
Here we derive the limiting null distributions for the new tests. In each case 
the test statistic  is a degenerate $U$-statistic of order four and rank two if the variables 
are independent, which implies that the limit distribution is the same as the distribution 
of a weighted sum of independent centered chisquare variables with one degree of freedom.
We discuss in detail the identification of these weights, which are known to be the non-zero 
eigenvalues of an integral operator associated with the (reduced) kernel of the $U$-statistic. 

Given that we now have several consistent tests that are all based on counts of patterns of length 
four, a quantitative comparison is of interest, not least from an applied point of view. Such
a comparison may be based on the behavior of the tests under alternatives. We consider 
various notions of efficiency and show that the local Bahadur and limiting Pitman efficiencies 
lead to the same value for each pair of tests. This value depends on the individual tests 
through the largest eigenvalues and the eigenfunctions of the respective integral operator.
These also determine the locally optimal direction of the alternative for each test; 
see Figure~\ref{fig:dir} below for a graphical representation of a  specific outcome.

We turn to a formal treatment.
Suppose that $z_1,\ldots,z_n$ is a two-dimensional data set, $z_i=(x_i,y_i)$ with 
$x_i,y_i\in \bR$, with no multiple values (ties) in the two coordinate data sets 
$x_1,\ldots,x_n$ and $y_1,\ldots,y_n$. Then the respective rank vectors $q_n,r_n$ 
with components 
\begin{equation}\label{eq:ranks}
  q_{n,i}:=\# \{1\le j\le n:\, x_j\le x_i\},\quad r_{n,i}:=\# \{1\le j\le n:\, y_j\le y_i\},   
\end{equation}
$i\in[n]:=\{1,2,\ldots,n\}$, both define an element of the group $\bS_n$ of permutations 
of $[n]$. By the permutation generated by the data set we mean the composition 
$\pi_n:=r_n\circ q_n^{-1}$. 
In contrast to the ranks, $\pi_n$ is invariant under reordering of $z_1,\ldots,z_n$. From 
a data analytic point of view, the transition from the set $\{(x_i,y_i):\, i\in [n]\}$ to
$\{(q_{n,i},r_{n,i}):\, i\in [n]\}=\{(i,\pi_{n,i}):\, i\in [n]\}$ corresponds to the transition
from the scatter plot to the rank or permutation plot.

Regarding  the data as the values of the first $n$ terms of a sequence $(Z_i)_{i\in\bN}$ of independent 
and identically distributed random variables, we write $Q_n,R_n,\Pi_n$ for the random elements of
$\bS_n$ associated with $Z_1,\ldots,Z_n$. We assume throughout that the distribution functions
$F_X,F_Y$ of the component variables $X_i,Y_i$ are continuous so that the above no-ties 
condition holds with probability~1. Using the invariance under the probability transform  
we may then further assume that the marginal distributions are uniform on the unit interval, 
so that the distribution function $C$ of the $Z$-variables is an element of the set $\cC$ of two-dimensional copulas.
The reduction from data values to ranks then implies that testing independence leads to the
(simple) hypothesis
\begin{equation*}
   H_0:\, C=C\ind, \ \text{ with }C\ind(x,y):=xy \; \text{ for } 0\le x,y\le 1,
\end{equation*} 
where the independence copula $C\ind$ is the distribution function associated with the uniform
distribution on the unit square. 
Under $H_0$ the random permutation $\Pi_n$ is uniformly distributed on $\bS_n$, which 
we abbreviate to $\Pi_n\sim \unif(\bS_n)$. In fact, it follows that $C=C\ind$ if this 
holds for all $n\in\bN$,
which may be interpreted as a characterization of the uniform distribution on the unit square. 

We are specifically interested in tests that use \emph{pattern counts}: Let
$\sigma=(\sigma_1,\ldots,\sigma_k)\in\bS_k$, $\pi=(\pi_1,\ldots,\pi_n)\in\bS_n$ with $k\le n$. 
We say that the pattern $\sigma$ appears in the permutation $\pi$ at
$A=\{m_1,\ldots,m_k\}\subset[n]$ if $\pi_{m_i}<\pi_{m_j}$ and $\sigma_i<\sigma_j$ 
are equivalent for all $i,j\in [k]$ with $i<j$ and $m_i<m_j$. We  abbreviate this to 
$\pi_A=\sigma$ and write
\begin{equation}\label{eq:pattern}
 N(\sigma,\pi) := \#\bigl\{A\subset [n]:\, \#A=k, \pi_A=\sigma\bigr\},\quad 
    t(\sigma,\pi) := {\textstyle\binom{n}{k}}^{-1} N(\sigma,\pi)
\end{equation}
for the absolute and relative number of occurrences of $\sigma\in\bS_k$ in $\pi\in\bS_n$.  
By a \emph{linear pattern statistic of degree $k$}, we mean a linear combination of pattern
counts $L_a(\Pi_n)$, with $L_a(\pi)=\sum_{\sigma\in \bS_k} a_\sigma t(\sigma,\pi)$ and
$a=(a_\sigma)_{\sigma\in\bS_k}\in\bR^{\bS_k}$. We will mainly deal with 0-1  valued
coefficient functions where, for some subset $A$ of $\bS_k$, $L_A(\pi)=\sum_{\sigma\in A} 
t(\sigma,\pi)$. Then the coefficient function is simply the indicator function of the set $A$. 

The classical example of a pattern-based test of independence is Kendall's tau, with a 
test statistic that may be written as a monotone function of the number of concordant pairs
or, equivalently, of $t(\sigma,\Pi_n)$ with $\sigma=(12)$. The question whether the use of longer
patterns leads to tests with better power properties has already been investigated
by~Hoeffding~\cite{Hoeffding}, who showed that with patterns of length five it is possible to 
obtain a test that is consistent against all alternatives $C\in\cC'$, where $\cC'$ denotes 
the set of all copulas with continuous densities, and that this is closely related 
to a characterization of $C\ind$ within $\cC'$. Yanagimoto~\cite{Yana} obtained such a 
characterization of $C\ind$ within $\cC'$ with patterns of length four, where $\cC'$ is 
now the set of all absolutely continuous copulas, and also showed that such a result is 
not possible with patterns of length three. This was later rediscovered by Bergsma and
Dassios~\cite{BergsmaDassios} with yet another $\cC'$, now including discrete distributions. 
Below we will refer to the resulting procedure as the BDY test.
The latter authors also provide a statistical motivation with jointly concordant and 
discordant pairs, connecting the BDY test to Kendall's tau; 
see also~\cite[Section 6.2]{DrtonHanShi}. 

The fact that the uniform distribution on the unit square can be characterized by pattern
probabilities is connected to a property known as \emph{quasirandomness} in discrete mathematics 
and theoretical computer science. Chan et al.\ \cite{Chan} found other (indicator) functions 
on $\bS_4$ that imply quasirandomness, and even obtained an associated characterization.
These can be used to construct further consistent tests of independence. The method of proof 
in~\cite{Chan} is quite different from that in~\cite{Hoeffding,Yana,BergsmaDassios}, it
relies on a tool from extremal combinatorics. This approach leads to a characterization of 
$C\ind$ within the class \emph{all} copulas, no smoothness 
condition is needed beyond continuity of the marginal distribution functions $F_X,F_Y$. 

Our main aim in the present paper is an analysis of the tests that are based on the 
additional characterizations found in~\cite{Chan}. For the BDY test the limit distribution 
under the null hypothesis has been described in~\cite{NWD}, and~\cite{DDB} considers 
the limit distribution under contiguous alternatives. In all cases, the 
respective test statistics can be written as 
$U$-statistics; hence, for an assessment based on the behavior as $n\to\infty$, the asymptotics 
of the latter are important. Indeed, a corresponding law of large number is essential for 
transforming a distributional characterization into a consistent test; 
see~\cite[Section 3]{Hoeffding}. We obtain the limit distributions under the null hypothesis 
for the new tests (Theorem~\ref{thm:asdistr}).

Next we compare the BDY and the new tests through various notions of efficiency:  
Local exact and local approximate Bahadur efficiency, limiting Pitman efficiency, and an efficiency
concept proposed by Gregory \cite{Gregory80} for comparing quadratic tests. 
We obtain explicit results for various copula families that are commonly considered as
suitable alternatives to the independence copula. 
Our analysis in the case of fixed alternatives is based on an extension of a large deviation result for 
degenerate $U$-statistics given in \cite{NikiPoni}. For the behavior of tests under local alternatives 
we use results from \cite{Gregory} on the corresponding distributional limits
of such $U$-statistics.

Quasirandomness leads to tests that are consistent, meaning that any form of dependence 
will be discovered if the sample size is large enough. Classical such `omnibus tests'
are the Cram\'er-von Mises and the Kolmogorov-Smirnov goodness-of-fit tests, where
the behavior under alternatives has been investigated intensively; see \cite{ShorackWellner}
and the references given there. Whereas in the goodness-of-fit situation the distributional limits of 
the  test statistics directly appear as `quadratic forms' in some
Gaussian process, specifically the Brownian bridge, for the `quadratic tests'
in the pattern-based situation we need to construct suitable Gaussian processes 
first. This is achieved via a Kac-Siegert representation
that is based on the spectral decomposition of the Hilbert--Schmidt operator arising in the
analysis of the degenerate $U$-statistics. Once this is in place,
we can follow the treatment of the Cram\'er-von Mises test given by Neuhaus~\cite{Neuhaus}. 
This includes an analysis of the local power of the individual 
tests as a function of the direction of the local alternatives; see e.g. \cite{Janssen},
\cite{MilStr}, and \cite[Section 2]{Neuhaus}.

Below we first collect some basic material in Section~\ref{sec:basics}. Sections~\ref{sec:distr}
and~\ref{sec:eff} contain our main results, on limit distributions and efficiencies.  
In the final section we discuss implementation aspects 
and augment the theoretical results  by a small simulation study.

In the present paper we consider pattern frequencies in connection with independence,
but patterns can also be used for a variety of other testing problems; see~\cite{BaGr}.
 
\section{Notation, basic concepts, and background.}\label{sec:basics}
We need material from different areas and subdivide this section accordingly. Throughout, we write 
$\eqdistr$ and $\todistr$ for equality in distribution and convergence in  distribution of 
random variables, and $\toweak$ denotes weak convergence of probability measures. 
Formally, a random variable $X$ with values in $(E,\cF)$ is defined on some background probability 
space $(\Omega,\cF,\bP)$, and its distribution $P$ is the push-forward of $\bP$ under the
measurable mapping $X$, hence a probability measure on $(E,\cF)$. In the notation we do not distinguish
between $\bP$ and $P$ (or $\bP_\theta$ and $P_\theta$) if the meaning is clear from the context.

\subsection{Permutations.}\label{subsec:perm}
We write $\bS:=\bigcup_{k=0}^\infty \bS_k$ for the set of all finite permutations 
and $|\pi|$ for the length (or degree) of $\pi$, i.e.\ $|\pi|=k$ if $\pi\in\bS_k$. 
In~\cite{Hopp} a pattern-based topology
for (deterministic) sequences $(\pi_n)_{n\in\bN}$ of permutations is introduced 
by requiring that  the sequences $(t(\sigma,\pi_n))_{n\in\bN}$ defined in~\eqref{eq:pattern}
converge for all $\sigma\in\bS$. If $|\pi_n|\to\infty$ then the limits are described
by permutons, or equivalently two-dimensional copulas, in the sense that, for some copula $C$, 
$\lim_{n\to\infty} t(\sigma,\pi_n) = t(\sigma, C)$ for all $\sigma\in\bS$. Here, 
$t(\sigma,C) = P_C(\Pi_k=\sigma)$ if $|\sigma|=k$,
where $\Pi_k$ is the random permutation generated by 
independent random vectors $Z_1,\ldots,Z_k$ with distribution function $C$. 
For each $k\in\bN$, the limiting relative frequencies sum to 1 over $\bS_k$, and it is
easy to see that the resulting probability distributions $\mu_k$ over $\bS_k$ are consistent
(or projective) in the sense that $\mu_k(\sigma)$ is the sum of all $\mu_{k+1}(\tau)$ 
with $\tau_{[k+1]\setminus\{i\}}=\sigma$ for some $i\in[k+1]$.
In fact, $t(\cdot,C)$ may be seen as the projective limit of $\mu_k$ as $k\to\infty$; 
see~\cite{GrMet} for details and the connection to Markov chain boundary theory.
 
On $\bS$ we consider two bijections, `$\inv$' and `$\rev$', given by $\inv(\pi)=\pi^{-1}$ 
and $\rev(\pi)=(\pi_k,\pi_{k-1},\ldots,\pi_1)=\pi\circ (k,k-1,\ldots,1)$ if
$\pi\in\bS_k$. These generate a group $G$ that acts on $\bS$, where $G$ is isomorphic to the 
dihedral group $D_4$. The eight elements of $G$ interact in an equivariant manner with 
pattern counting in the sense that $t(g.\sigma,g.\pi)=t(\sigma,\pi)$ for all $g\in G$ 
and $\sigma,\pi\in\bS$, see also~\cite[p.922]{NWD}. 
The action of $G$ on $\bS$ can be extended to the limits so that $\pi_n\to C$ 
implies $g.\pi_n\to g.C$ for all $g\in G$, where
$g.C$ is the copula associated with $(Y,X)$ if $g=\inv$ and with $(1-X,Y)$ if $g=\rev$.
Clearly, for the independence copula, $g.C\ind=C\ind$ for all $g\in G$, so that the hypothesis
of independence is invariant under $G$. From a statistical point of view we may thus require 
that, at least in a nonparametric setup,  a linear pattern statistic $L_a$ is such that
$L_a(g.\Pi_n)=L_a(\Pi_n)$ for all $n\in\bN$.
If $L_a$ has degree~$k$ then this is equivalent to $a(g.\pi)=a(\pi)$ for all $g\in G$ and
$\sigma\in\bS_k$. In other words, the coefficient function $a$ on $\bS_k$ is constant on the orbits 
induced by the group.

\subsection{Quasirandomness.}\label{subsec:quasi}
We next address the connection to a concept from discrete mathematics.
A (deterministic) sequence $(\pi_n)_{n\in\bN}$ of permutations with $|\pi_n|\to\infty$ 
is said to be \emph{quasirandom} if, for all $k\in\bN$,
\begin{equation}\label{eq:quasi}
    \lim_{n\to\infty} t(\sigma,\pi_n) = \frac{1}{k!}\quad \text{for all }\sigma\in\bS_k.
\end{equation}
Thus, the sequence converges in the pattern frequency topology, and the limit is described
by the independence copula $C\ind$. A random sequence $(\Pi_n)_{n\in\bN}$ generated 
by an i.i.d.\ sequence $(Z_n)_{n\in\bN}$ of random variables with distribution function $C\ind$
has this property with probability~1.
In \eqref{eq:quasi} the degree $k$ of the permutations appears as a parameter and
by projectivity, if~\eqref{eq:quasi} holds for some $k$ then it does so for all $j<k$.
Remarkably, a `forward extension' is valid if the limit is $C\ind$ as in~ \eqref{eq:quasi}:
If it holds for $k=4$ then  it is true for all $k\in\bN$. Hence, 
if all 24 patterns of length four appear with the same limit frequency then the sequence is 
quasirandom. In the probabilistic context,
$t(\sigma,\Pi_n) \to 1/24$ a.s.\ for all $\sigma\in\bS_4$ implies independence of the components
$X_i$ and $Y_i$ of $Z_i$. Moreover, $k=4$ is the smallest possible value with this property.

The main result underlying the BDY  test~\cite[Theorem 1]{BergsmaDassios}
implies that the vector $(t(\sigma,\Pi_n))_{\sigma\in\bS_4}$ may even be reduced to
$L_B(\Pi_n):=\sum_{\sigma\in B} t(\sigma,\Pi_n)$ for a specific set
$B\subset \bS_4$. Four more such subsets were found in~\cite{Chan}.
The five sets are given in Table~\ref{tab:perm4} as $B,C,D,E,F$. 
Note that $B$ and $C$ are both subsets of $F$. It was further shown in~\cite{Chan}
that these, together with their complements, exploit all possibilities. In particular,
in the statistical application,
\begin{equation*}
   L_A(\Pi_n)\to \frac{\# A}{24}\  \text{ a.s.\ as }n\to\infty\quad\Longleftrightarrow
           \quad C=C\ind
\end{equation*}
for such sets $A$, and   
\begin{equation}\label{eq:Chancharac}
  E_C L_A(\Pi_n)\,\ge \, \frac{\# A}{24}\quad\text{for }\,A\in \{B,C,F\},\quad
  E_C L_A(\Pi_n)\,\le \, \frac{\# A}{24}\quad\text{for }\,A\in \{D,E\},  
\end{equation}
where for each specific $A\in \{B,C,D,E,F\}$ equality in holds in \eqref{eq:Chancharac} 
if and only if $C=C\ind$. Motivated by these characterizations we define the test sequences $T^A=(T^A_n)_{n\in\bN}$
with test statistics 
\begin{equation}\label{eq:deftests}
   T^A_n=T^A_n(Z_1,\ldots,Z_n)= \begin{cases}L_A(\Pi_n)-\# A/24, &\text{if } A\in\{B,C,F\},\\
                                                      \# A/24 - L_A(\Pi_n), &\text{if } A\in\{D,E\},
                 \end{cases}
\end{equation}
rejecting the hypothesis if the test statistic exceeds a certain critical value. The $T_n^A$ are
$U$-statistics. A short review of the asymptotic theory of this type of statistics
is given below in Subsection \ref{subsec:Ustat}.

\begin{table}
\begin{tabular}{rllcrllcrllcrll}
\noalign{\hrule}
\noalign{\vspace{2mm}}
    1& \textcolor{black}{1234} & BCF  &&   7 & \textcolor{blue}{2134} & BF && 
  13& \textcolor{green}{3124} & D      && 19 &\textcolor{magenta}{4123} & CF \\
    2& \textcolor{blue}{1243} & BF     &&   8 & \textcolor{red}{2143} & BCF && 
  14& \textcolor{violet}{3142} & DE    && 20 &\textcolor{green}{4132} & E\\
    3& \textcolor{cyan}{1324} & DE        && 9 &  \textcolor{green}{2314} & E && 
  15&  \textcolor{magenta}{3214} & CF && 21 & \textcolor{green}{4213} & D \\
    4&  \textcolor{green}{1342} & D          && 10 & \textcolor{magenta}{2341} &CF  && 
  16&  \textcolor{green}{3241} & E && 22 & \textcolor{cyan}{4231} & DE \\
    5&  \textcolor{green}{1423} &  E          && 11 &  \textcolor{violet}{2413} & DE  && 
  17&  \textcolor{red}{3412} & BCF && 23  & \textcolor{blue}{4312} & BF \\
    6&  \textcolor{magenta}{1432} & CF      && 12  &  \textcolor{green}{2431} & D && 
  18&  \textcolor{blue}{3421} & BF && 24& \textcolor{black}{4321} &BCF \\
\noalign{\vspace{1mm}}
\noalign{\hrule}
\noalign{\vspace{2mm}}
\end{tabular}
\caption{\label{tab:perm4} Elements of $\bS_4$, orbits,  and subset definitions (see text)}
\vspace{-5mm}
\end{table}

Yanagimoto's characterization result~\cite{Yana} is based on a function $L_0$ of the 
two marginal ranks defined in~\eqref{eq:ranks}, with 
$L_0(q,r)=1_H(q)(1_H(r)- 2\cdot 1_J(r) + 2\cdot 1_K(r))- 1/9$ and 
$H,J,K$ the sets of permutations $\pi=(ijkl)\in\bS_4$ characterized by $(i<l)\wedge (j<l)$, 
by $(i<l)\wedge (k<l)$ and by $(i<k)\wedge (j<k)$ respectively. 
This deviates from the statistics used here. 
However, $\Pi_n$ is easily seen to be sufficient for the pair $(Q_n,R_n)$,
and `Rao-Blackwellization' motivates the test statistic
\begin{equation*}
       L(\pi):=\sum_{\{(q,r):\pi=r\circ q^{-1}\}}L_0(q,r),
\end{equation*}
and this leads to a procedure that is equivalent to the BDY test.

For the invariance of the respective tests under the dihedral group
the relation to the orbits is of importance. In
Table~\ref{tab:perm4} the seven orbits are indicated by
different colors. The set $B$ is the union of the black, blue and red sets, hence the BDY test 
has the desired equivariance. The analogous statement holds for the sets $C$ and $F$, but not 
for the remaining $D$ and $E$. Thus, the analogues of the BDY test based on these two sets 
are not invariant under the dihedral group, but both are invariant under the transformation $g=\rev$,
and transformed into each other by $g=\inv$. In contrast to these two tests, the test based on the linear pattern
statistic
\begin{equation}\label{eq:TDE}
  T_n^{\DE}\,:=\,T_n^D\,+\,T_n^E
\end{equation}
is invariant under the dihedral group, and will also be considered below. 

The results in~\cite{Chan} imply that at least eight patterns of length four are needed for 
quasi\-randomness if we allow coefficients 0 and 1 only in the linear pattern statistic $L$.
Very recently, Crudele et al.~\cite{CDN} showed that the following weighted sum of only six 
permutations forces quasirandomness,
\begin{equation*}
     \rho^* := 123 + 321 +  2143 + 3412 + \frac{1}{2}\bigl(2413+3142\bigr).
\end{equation*}
This can be used to obtain yet another consistent (and invariant) test of independence.  
The result is also of interest from a geometric point of view:
The set of linear statistics $L$ with non-negative coefficients that force quasirandomness 
in the sense that 
\begin{equation}\label{eq:char}
 E_C L > E_{C\ind}L\quad  
                \text{ for all }C\not= C\ind
\end{equation}  
is easily seen to be a convex cone (see also~\eqref{eq:TDE}). The projectivity  of the distributions on 
$\bS_k$ with mass functions $\bS_k\ni\sigma\mapsto t(\sigma,C)$ can be 
used to 
show that we have $E_C \rho^*=\frac{1}{4}E_C L_{a^*}$ 
for all $C\in\cC$,  where $a^*\in\bR^{24}$ is given by 
\begin{equation}\label{eq:defastar}
    a^* = (4, 2, 2, 1, 1, 1, 2, 4, 1, 1, 2, 1, 1, 2, 1, 1, 4, 2, 1, 1, 1, 2, 2, 4).
\end{equation}
The cone spanned by the five subsets in Table~\ref{tab:perm4} (with the modification as in~\eqref{eq:deftests})
has dimension five, including $a^*$ increases this  to six.  Moreover we may also subtract the value~1 from 
each of the components of $a^*$ without violating~\eqref{eq:char},
thereby arriving at another set of twelve elements of $\bS_4$ that forces quasirandomness if
we allow coefficients other than 0 or 1.

\subsection{$U$-statistics.}\label{subsec:Ustat}
We give an outline of that part of the theory of $U$-statistics that we will need below;
for details see e.g.\ \cite{DynMan}, \cite{KoroljukBorovskich} or the original paper~\cite{Hoeffding} by Hoeffding.   The functional analytic aspects are treated in~\cite{Conway}, for example.

Let $(Z_i)_{i\in\bN}$ be a sequence of independent and identically distributed random 
variables with distribution $P$ and values in some measurable space $(E,\cF)$.
Let $k\ge 2$ be an integer, and let $h:E^k\to \bR$ be a measurable, $P$-square integrable function
that is symmetric in its variables.
Then,  
\begin{equation}\label{eq:Ustat}
 U_{h,n} \,:=\,U_{h,n}(Z_1,\ldots,Z_n)\,=\, {\textstyle\binom{n}{k}}^{-1} \sum_{1\le i_1<\cdots < i_k\le n}
              h(Z_{i_1},\ldots,Z_{i_k}),\quad n\ge k,
\end{equation}
is the sequence of $U$-statistics with kernel $h$, and $k$ is the order of the kernel (or of the $U$-statistics).  
We may assume that $h_0:=\bE h(Z_1,\ldots,Z_k)=0$, passing from $h$ to $h-h_0$ if necessary. When dealing
with the limit distribution of these $U$-statistics the functions $h_i:E^i\to \bR$, $i=1,2$,
given by 
\begin{equation}\label{eq:h1h2}
 h_1(z_1):= \bE h(z_1,Z_2,\ldots,Z_k),\quad  h_2(z_1,z_2):= \bE  h(z_1,z_2,Z_3,\ldots,Z_k),
\end{equation}
$z_1.z_2\in E$, are important. 
In the present context, the special case with $h_1=0$ $P$-almost surely and 
$\sigma_2^2(h) := \bE h_2(Z_1,Z_2)^2 >0$ is of main interest. Then the $U$-statistics 
$U_{h,n}$ are called degenerate, with degenerate kernel of rank 2. We then
consider the associated $U$-statistics
\begin{equation*}
 U_{h_2,n} \,:=\,U_{h_2,n}(Z_1,\ldots,Z_n)\,=\  {\textstyle\binom{n}{2}}^{-1} \sum_{1\le i<j\le n}
              h_2(Z_{i},Z_{j}),\quad n\ge 2,
\end{equation*}
with degenerate kernel $h_2$ of order 2 and rank 2, and define an operator 
$K:\bH\to \bH$, $\bH:=L^2(E,\cF,P)$, by
\begin{equation}\label{eq:specK}
   K f(x)  :=  \int h_2(x,y)  f(y)\, P(dy),\quad f\in\bH,~x\in E,
\end{equation}
see e.g. \cite[Sec.\,5.5.2]{Serfling}. This is the  Hilbert--Schmidt operator associated with the kernel $h_2$.
It has a finite or countably infinite set $\varphi_j$, $j\ge 1$, of orthonormal eigenfunctions
$\varphi_j\in \bH$, with corresponding non-zero eigenvalues $\lambda_j$, $j\ge 1$, and spectral
representation
\begin{equation}\label{eq:Kspec}
    K f =\sum_{i\ge 1} \lambda_i \langle  f,\varphi_i\rangle \, \varphi_i\quad\text{for all }   f\in\bH.
\end{equation}
We then have convergence in distribution together with a representation of the limit,
\begin{equation}\label{eq:Ustatconv}
  nU_{h_2,n}\todistr U \ \text{ and }\ \cL(U)=\mu(K):=\cL\biggl(\,\sum_{j\ge 1} \lambda_j \bigl(\xi_j^2-1\bigr)\biggr),
\end{equation}
with  independent standard normal random variables $\xi_j$, $j\in \bN$.
Note that the distribution function associated with $\mu(K)$ is continuous and strictly increasing.
Note also that $U_{h_2,n}$ is called a `quadratic' $U$-statistic if the non-zero eigenvalues of $K$ are all positive.  

If $(S_n)_{n\in\bN}$ is a sequence
of statistics $S_n=S_n(Z_1,\ldots,Z_n)$ that is asymptotically equivalent to the sequence of $U$-statistics
$U_{h_2,n}$ in the sense that $nS_n\,-\,nU_{h_2,n}\to 0$ in probability, then $\mu(K)$ is also the limit distribution 
for $S_n$ as $n\to\infty$. The sequence $S_n=U_{h,n}/\binom{k}{2}$, $n\ge k$, of scaled $U$-statistics 
\eqref{eq:Ustat} is the most important example; see \cite[p.\,190]{Serfling}.  

To see the connection with pattern counting let $\sigma\in\bS_k$, $E=[0,1]^2$,
and define the kernel $h=h_\sigma=h_\sigma(z_1,\ldots,z_k)$ to be 1 if the pattern associated with 
its arguments is equal to $\sigma$, and 0 otherwise. This leads to $U_n=t(\sigma,\Pi_n)$ and, by linearity, 
$L_a(\Pi_n)$ is the $U$-statistic associated with $h^a=\sum_{\sigma\in\bS_k} a_\sigma h_\sigma$. 
Thus, the kernels of interest in the present application are indeed bounded. This is enough 
for a strong law of large numbers, meaning that, $P$-almost surely, $U_n\to EU_k$ as $n\to\infty$. 

Let $r_k(y_1,\ldots,y_k)$ and $q_k(x_1,\ldots,x_k)$ be the rank vectors 
of the respective coordinates of the input $z_1=(x_1,y_1),\ldots,z_k=(x_k,y_k)$. 
Then the kernel $h^a$ may be written as 
\begin{equation}\label{eq:hfac}
   h^a(z_1,\ldots,z_k) 
          = \tilde h^a\bigl(r_k(y_1,\ldots,y_k)\circ q_k(x_1,\ldots,x_k)^{-1}\bigr)
\end{equation}
with $\tilde h^a:\bS_k \to \bR$ given by $\tilde h^a(\sigma)=a_\sigma$ for all $\sigma\in\bS_k$. 
Equation~\eqref{eq:hfac} relates the patterns in $(Z_i)_{i\in\bN}$ to permutations in the component sequences   
$(X_i)_{i\in\bN}$ and $(Y_i)_{i\in\bN}$. Under the hypotheses of independence the $Z$-variables are uniformly distributed
on the unit square, which implies that the components are uniformly distributed on the unit interval. The operators $K$ 
thus lead us to consider the (real) Hilbert spaces  
\begin{equation*}
\bH_1:=L^2\bigl([0,1],\cF_1,\unif([0,1])\bigr), 
     \quad \bH_2:=L^2\bigl([0,1]^2,\cF_2,\unif([0,1]^2)\bigr),
\end{equation*}
where $\cF_1,\cF_2$ are the respective Borel $\sigma$-fields. 
The second space may be written as the tensor product $\bH_2=\bH_1\otimes\bH_1$; 
see e.g.\ \cite[Lemma II.4.8]{Conway}. This implies, for example, 
that $\{\varphi_i\otimes \psi_j:\, i,j\in \bN\}$ is a complete orthonormal system 
for $\bH_2$ if $\{\varphi_i:\, i\in\bN\}$ and $\{\psi_j:\, j\in\bN\}$ are both
complete orthonormal systems for $\bH_1$. 

Let $T^A$, $A\in \{B,C,D,E,F,\DE\}$, be the family of tests introduced in the previous subsection, where we choose
the critical value at level $\alpha$ to be the respective upper $\alpha$-quantile $t^A_{n,\alpha}$ of the distribution 
of $T^A_n$ under the hypothesis of independence.
Using the above we will derive in Section \ref{sec:distr} the limiting null distribution of $T^A_n$,
which is a $U$-statistic of order four. 


\subsection{Gaussian processes.}\label{subsec:processes} 
We recall that a Gaussian process $W=(W_t)_{t\in E}$ indexed by a set $E$ 
is a family of random variables 
with the property that, for all $k\in\bN$ and $t_1,\ldots,t_k\in E$, the random vector $(W_{t_1},\ldots,W_{t_k})$ has a 
$k$-dimensional normal distribution. As $W$ may be regarded as a random variable with values in the 
product space $\bR^E$, its distribution is characterized by its mean function $\mu_W(t)=EW_t$ and 
covariance function $\rho_W(s,t)=\cov(W_s,W_t)$, $s,t\in E$. We say that $W$ is centered if $\mu_W\equiv 0$.  

Here we will always have $E=[0,1]$ or $E=[0,1]^2$ and $\rho_W$ will be continuous on $E\times E$.
Then the paths $t\mapsto W_t$ are mean square continuous with probability 1, and $W$ 
may be seen as a variable with values in $\bH=\bH_1$ or $\bH=\bH_2$. Every 
orthonormal base $(\phi_i)_{i\in I}$ of~$\bH$ then provides a representation 
$W= \sum_{i\in I} \langle W,\phi_i\rangle\, \phi_i$,
with convergence and inner product referring to the Hilbert space $\bH$. A special choice is of particular 
importance: The covariance function of the process defines a Hilbert--Schmidt operator 
\begin{equation*}
K:\bH\to\bH, \  \  (Kf)(x)=\int_E  \rho_W(x,y)f(y)\, dy,
\end{equation*} 
and its spectral decomposition $(\lambda_i,\psi_i)$, $i\in\bN$, leads to what is known as 
the \emph{Kac--Siegert representation} or \emph{Karhunen--Lo\`eve expansion} of the process,
\begin{equation}\label{eq:KacSiegert}
             W = \sum_{i \ge 1} \sqrt{\lambda_i}\, \xi_i\,\psi_i  ,
\end{equation}  
with independent standard normals $\xi_i$, $i\in\bN$;
in particular, $\|W\|^2 = \sum_{i\ge 1} \lambda_i \xi_i^2$. 
Hence, apart from a deterministic shift by the trace $\sum_{i \ge 1} \lambda_i$ of $K$, 
the limiting null distribution in~\eqref{eq:Ustatconv} is the same as that of $\|W\|^2$; see also~\cite[p.\,121]{Gregory}. 

The above may be reverted: If $K$ is given, as in Section~\ref{subsec:Ustat}, where
$K$ is assumed to be positive and of trace class, i.e. the non-zero eigenvalues $\lambda_j$, $j\ge 1$, of $K$
are positive, and $\sum_{j\ge 1} \lambda_j < \infty$, we may 
use~\eqref{eq:KacSiegert} to obtain a centered Gaussian process $W$ that has $K$ as the 
Hilbert--Schmidt operator associated with its covariance function $\rho_W$. 

Finally, suppose that $X$ and $Y$ are processes with paths in $\bH_1$, with Kac-Siegert representations 
$X \eqdistr \sum_{j\ge 1} \sqrt{\vphantom{\lambda_j} \eta_j} \, \xi^X_j\,\varphi_j$,
$Y \eqdistr \sum_{k\ge 1 } \sqrt{\vphantom{\lambda_j}\mu_k} \, \xi^Y_k \,\psi_k$,
and covariance functions $\rho_X,\rho_Y$ respectively. We can then combine these  
into a process $Z$ with paths in $\bH_2=\bH_1\otimes\bH_1$ and covariance 
function $\;\rho_Z=\rho_X\otimes\rho_Y$ via 
$Z= \sum_{j,k\ge 1}\sqrt{\vphantom{\lambda_j}\eta_j\mu_k} \, \xi_{j,k} \,\varphi_j\otimes\psi_k$, 
using a family $\{\xi_{j,k}:\, j,k\ge 1\}$ of independent standard normals. 

Using~\eqref{eq:Ustatconv} we obtain the limiting null distribution of the linear pattern statistic
$T_n^A$, with $A\in \{B,C,D,E,F,\DE\}$, through the non-zero eigenvalues of the respective Hilbert--Schmidt operator $K=K(A)$. 
One way to achieve this is to find a centered Gaussian process where the spectral decomposition of the 
integral operator  $K$ associated with its covariance function is known, see 
Section~\ref{sec:distr}. In  Section~\ref{sec:eff}, we will see 
that, in the context of local alternatives, the eigenfunctions are also important.

A basic reference for the above is the monograph~\cite{JansonGHS}, to which we refer for details, 
proofs and additional material; 
see also~\cite[Section 12.3]{vdV}. The process view also relates the `chaos' decomposition of Gaussian 
processes to the Hoeffding decomposition of $U$-statistics and thus provides a background for classical results such as~\eqref{eq:Ustatconv}. It can further serve as a framework for the treatment of contiguous
alternatives in connection with efficiency considerations.

\section{Limiting null distributions}\label{sec:distr}
We begin with a general outline.
The main results in this section are Theorems~\ref{thm:asdistr} and~\ref{thm:asdistrDE}, 
which provide representations for the respective limiting null distributions for the tests 
introduced in Section~\ref{subsec:quasi}.
Such results can be used to obtain critical values that are asymptotically correct. 
The representations are based on the spectral structure of the associated Hilbert--Schmidt 
operators $K$.  It turns out that in five of the  six cases $K$ can be written as a  tensor product 
as explained in Section~\ref{subsec:Ustat}. The spectral representations for the factors are given 
in Proposition~\ref{prop:eig_g1}. Only three different factor operators appear. The first of
these  is related to the Cram\'er--von Mises goodness-of-fit statistic, as already noted in~\cite{NWD},
for the second we find a similar connection  to Watson's goodness-of-fit test for distributions 
on the unit circle. For the third factor we use a decomposition of the operator $K$ of interest as 
$K=K_1+K_2$, where  the spectral decomposition of $K_1$ is known and  $K_2$ is a projection. 
The results for the factors then lead to the spectral representation of the  products in Proposition~\ref{prop:eigstruct1}.  For the remaining case
we do not have such a simple tensor product  decomposition, and some extra work is needed, 
leading to  Proposition~\ref{prop:eig_g4}.  As with the third factor in Proposition~\ref{prop:eig_g1}, but now 
on $\bH_2$ rather than on $\bH_1$, the first step is to write the operator of interest  as a `perturbation' 
of an operator with known spectral decomposition.

In detail we consider the limit distributions under the hypothesis of independence for the test statistics 
$T_n^A$  based on the sets $B,C,D,E,F$ in Table~\ref{tab:perm4}, see~\eqref{eq:deftests}, but also
that for the test statistic $T_n^{\DE}=T_n^D + T_n^E$. 
For the BDY test the limiting distribution and the behavior under contiguous
alternatives have already been discussed in~\cite{NWD} and \cite{DDB}. Here too, our analysis
relies on the asymptotics of degenerate  $U$-statistics, which we obtain via the respective
$h$-functions in~\eqref{eq:h1h2} and the eigenstructure
of the associated Hilbert--Schmidt operators $K$ in~\eqref{eq:specK}.  

For a given significance level $\alpha\in (0,1)$ the hypothesis $H_0$ of independence is rejected
if $T_n^A>t_{n,\alpha}^A$, with $t_{n,\alpha}^A$ as the upper $\alpha$ quantile of $T_n^A$ in the case where the
hypothesis is true.
The BDY test corresponds to the set~$B$; see \cite{BergsmaDassios,EZLeng,NWD} and  the
references given there for variants and representations of the test statistic. 
We recall that $T_n^A$ is a $U$-statistic of order~$4$ and that  the associated functions 
$h_1^A$, $h_2^A$ are given by
\begin{equation*}
h_1^A(z_1)=\bE [T_4^A|Z_1=z_1], \quad h_2^A(z_1,z_2) = \bE [T_4^A|Z_1=z_1,Z_2=z_2],
\end{equation*}
$z_1,z_2\in[0,1]^2$, where the expectation refers to a sample $Z_1,Z_2,Z_3,Z_4$ 
of random vectors with independent components. We write $K(A)$ for the operator with kernel  
$h_2^A$, $A\in \{B,C,D,E,F,\DE\}$.  

\begin{table}[t]

\renewcommand{\arraystretch}{1.1}
\begin{tabular}{lcclcclcclc}
\noalign{\hrule}

\noalign{\vspace{2mm}}
    1234 & $\frac{1}{2}(1-v)^2$  & \  & 2134 & $0$ & $\ $ 
              &  3124 & $0$  &\ & 4123 & $0$  \\ 
    1243 & $\frac{1}{2}(1-v)^2$  & \ & 2143 & $0$ & $\  $
              &  3142 & $0$  &\  & 4132 & $0$   \\
    1324 & $(1-v)(v-u)$  & \ & 2314 & $u(1-v)$ & \  
              & 3214 & $0$  &\ & 4213 & $0$ \\
    1342 & $(1-v)(v-u)$  &\  & 2341 & $u(1-v)$ & \  
              & 3241 & $0$  &\  & 4231 & $0$ \\
    1423 & $\frac{1}{2}(v-u)^2$  & \  & 2413 & $u(v-u)$ & \
              &  3412 & $\frac{1}{2}u^2$  &\ & 4312 & $0$ \\
    1432 & $\frac{1}{2}(v-u)^2$  & \ & 2431 & $u(v-u)$ & \  
              &  3421 & $\frac{1}{2}u^2$  &\  & 4321 & $0$\\
\noalign{\vspace{1mm}}
\noalign{\hrule}
\noalign{\vspace{2mm}}
\end{tabular}
 
\begin{tabular}{lcclcclcclc}
\noalign{\hrule}
\noalign{\vspace{2mm}}
    1234 & $0$ &$\  $& 2134 & $\frac{1}{2}(1-u)^2$ &$\ $ 
              &  3124 &$(1-u)(u-v)$ &$\ $&  4123  &$\frac{1}{2}(u-v)^2$ \\ 
    1243 & $0$  & $\  $& 2143 &$\frac{1}{2}(1-u)^2$ &$\ $  
              &  3142  &$(1-u)(u-v)$ &$\ $&  4132  &$\frac{1}{2}(u-v)^2$ \\
    1324 & $0$ &$\ $& 2314 & $0$ & \  
              & 3214 &$v(1-u)$ &\ &4213 &$v(u-v)$ \\
    1342 & $0$ &$\  $& 2341 & $0$& $\ $ 
              & 3241&$v(1-u)$ &$\ $& 4231  &$v(u-v)$ \\
    1423  & $0$ &$\ $& 2413  & $0$ &$\  $ 
              &  3412  &$0$ &$\  $& 4312 &$\frac{1}{2}v^2$ \\
    1432 & $0$ &$\ $ &2431 & $0$ &$\ $
              &  3421 &$0$ &$\ $& 4321 &$\frac{1}{2}v^2$ \\
\noalign{\vspace{1mm}}
\noalign{\hrule}
\noalign{\vspace{2mm}}
\end{tabular}
\caption{\label{tab:h2}$\cL[Q_4=\pi|X_1=u,X_2=v]$ for $u< v$ (top) and $u>v$ (bottom)}
\end{table}

Let $X_1,X_2,X_3,X_4$ be independent and uniformly 
distributed on the unit interval. Then $Q_4\sim\unif(\bS_4)$, and elementary arguments provide 
the conditional mass function $p_{u,v}$ of $\cL[Q_4|X_1=u,X_2=v]$, as summarized in 
Table~\ref{tab:h2}. Of course, $\cL[R_4|Y_1=u,Y_2=v]=\cL[Q_4|X_1=u,X_2=v]$ and,
using independence, the conditional distribution of $\Pi_4$ can be calculated 
via
\begin{equation*}
 P[\Pi_4=\pi|Z_1=(x_1,y_1),Z_2=(x_2,y_2)] 
     = \sum_{\sigma\in\bS_4} p_{y_1,y_2}(\pi\circ\sigma)\; p_{x_1,x_2}(\sigma)
       \text{ for all}\,\pi\in\bS_4.
\end{equation*} 
From these we obtain $h_2^A$ by summing over $\pi\in A$ and centering. The five
functions $h_2^A$, $A\in \{B,C,D,E,F\}$, all factorize, in the sense that for all $z_1=(x_1,y_1)$ and $z_2=(x_2,y_2)$ 
in the unit square,
\begin{align*}
   h^B_2(x_1,y_1,x_2,y_2) \,&=\, \frac{1}{6}\, g_1(x_1,x_2)g_1(y_1,y_2), \\ 
   h^C_2(x_1,y_1,x_2,y_2) \,&=\, \frac{1}{6}\, g_2(x_1,x_2)g_2(y_1,y_2), \\ 
   h^D_2(x_1,y_1,x_2,y_2) \,&=\, \frac{1}{6}\, g_2(x_1,x_2)g_1(y_1,y_2), \\ 
   h^E_2(x_1,y_1,x_2,y_2) \,&=\, \frac{1}{6}\, g_1(x_1,x_2)g_2(y_1,y_2), \\ 
   h^F_2(x_1,y_1,x_2,y_2) \,&=\, \frac{1}{2}\,g_3(x_1,x_2)g_3(y_1,y_2), \\
   \intertext{with}
   g_1(u,v)\,  &= \, 3u^2+3v^2-6u\vee v + 2, \\ 
   g_2(u,v)\,  &= \, 6u^2+6v^2-12 u v + 6u\wedge v- 6 u\vee v + 1, \\ 
  g_3(u,v)\,  &= \, 3u^2+3v^2-4uv +2u\wedge v -4 u\vee v+ 1
\end{align*}
for $0\le u,v\le 1$. For  $A=B$ this has already been obtained in~\cite{NWD}.
The invariance properties stated at the end of 
Section~\ref{sec:basics} are reflected by the 
symmetry of the factorizations for $B$, $C$, and $F$.
Note that swapping $x$- and $y$-components relates the functions for the remaining sets $D$ and $E$.
As  
\begin{equation*}
  h^{\DE}_2(x_1,y_1,x_2,y_2) \,=\, \frac{1}{6} \left (g_2(x_1,x_2)g_1(y_1,y_2) + g_1(x_1,x_2)g_2(y_1,y_2)\right ),
\end{equation*}
the invariance properties also hold for $A=\DE$. 
Obviously, $h_1$ can be obtained from $h_2$ by integrating over $Z_2$. It turns out that
$h_1^A\equiv 0$ almost surely under independence in all six cases, 
which shows that the $U$-statistics are indeed degenerate.

We next determine the spectral decomposition \eqref{eq:Kspec} of the operator 
$K_i:\bH_1\to \bH_1$ defined by
$(K_i f)(u)=\int_0^1 g_i(u,v) f(v)\, dv$, $ f\in \bH_1$, $i=1,2,3$.
The function $u\mapsto (K_i f)(u)$ is continuous for all $ f\in\bH_1$. 
Hence, if $(\lambda_j,\varphi_j)$ satisfies the eigenvalue relation 
$K_i \varphi_j = \lambda_j \varphi_j$ with non-zero
$\lambda_j$, then there is a continuous version of $\varphi_j:[0,1]\rightarrow \bR$ such that
\begin{equation*}
  \int_0^1 g_i(u,v)\,\varphi_j(v)\,dv\,=\,\lambda_j\varphi_j(u)\quad\text{for all }\,u \in [0,1].
\end{equation*}
It then even follows that this function $\varphi_j$ is twice continuously differentiable.   
It is easily verified that the Hilbert--Schmidt operators $K_i$ are positive in the sense that 
$\langle K_if,f\rangle \,\ge 0$ for all $f\in \bH_1$; this implies that the non-zero eigenvalues
of these operators are positive. In fact, the $K_i$ are covariance operators of
special Gaussian processes related to the classical Brownian bridge; see 
Section~\ref{subsec:processes} and Remark~\ref{rem:repr} below.
Further, as the kernels $g_i$ are symmetric and continuous, 
it follows from Mercer's theorem (see \cite[Theorem 11.6.7]{Dieudonne}) that  
\begin{equation*}
  g_i(u,v)\,=\,\sum_{j\ge 1} \lambda_j \varphi_j(u)\varphi_j(v),\quad 0\le u,v\le 1,
\end{equation*}
where the series converges absolutely and uniformly on the unit square. As a consequence, 
we have the trace formula,
\begin{equation}\label{eq:trace}
  \int_0^1 g_i(u,u)\,du \ = \ \sum_{j\ge 1} \lambda_j. 
\end{equation}

Proposition \ref{prop:eig_g1} summarizes the spectral decompositions of the operators $K_i$, $i=1,2,3$. 
Its first part is known, but in view of the connection to Gaussian processes we include the proof. 

\begin{proposition}\label{prop:eig_g1}
\emph{(a)} \cite[Lemma 4.3]{NWD} The non-zero eigenvalues of $K_1$ are
$\lambda_j=6/(\pi^2j^2)$, $j\in\bN$, all with multiplicity~1.
A corresponding set of orthonormal eigenfunctions is $\{\varphi_j:\, j\in\bN\}$, 
with $\varphi_j(u)=\sqrt{2}\cos(\pi j u)$. 

\vspace{.7mm}
\noindent
\emph{(b)} The non-zero eigenvalues of $K_2$ are $\lambda_j=3/(\pi^2j^2)$, $j\in\bN$, 
all with multiplicity~2.
A corresponding set of orthonormal eigenfunctions is $\{\varphi_j:\, j\in\bN\}\cup\{\psi_j:\, j\in\bN\}$, 
with $\varphi_j(u)=\sqrt{2}\cos(2\pi j u)$ and  $\psi_j(u)=\sqrt{2}\sin(2\pi j u)$. 

\vspace{.7mm}
\noindent
\emph{(c)} The operator $K_3$ has the simple eigenvalues 
$\lambda_j=3/(2\pi^2j^2)$, $j\in\bN$, with associated orthonormal eigenfunctions
$\varphi_j(u)=\sqrt{2}\cos(2\pi j u)$, $j\in\bN$, and the simple
eigenvalues $\mu_j=1/\gamma_j$, $j\in\bN$, where $\gamma_j\in \left (2\pi^2(j-1)^2/3,2\pi^2j^2/3\right )$, $j\in\bN$,
are the positive zeros of the function
\begin{equation}\label{eq:omega}
   \omega(z)\,=\,2/3\,+\,\sqrt{3z/2}\cot(\sqrt{3z/2})/3,\quad z\in\bR, 
\end{equation}
with associated orthonormal eigenfunctions
\begin{equation}\label{eq:Fourier}
  \chi_j(u)\,=\,c_j\left (\cos\bigl (\sqrt{6/\mu_j}\,u\bigr )\ + \ \sqrt{8\mu_j/3}\,\sin \bigl (\sqrt{6/\mu_j}\,u\bigr )\right ),\quad 0\le u \le 1,
\end{equation}
where~ $c_j\,=\,(1/2\,+\,2\mu_j)^{-1/2}$~ for $j\in \bN$.
\end{proposition}

\begin{proof} (a) We note that $g_1$ may be written as 
\begin{equation*}
g_1(u,v)\,  = \,6 \,\Tilde g_1(u,u),\quad\text{with }\,\Tilde g_1(u,v)\,=\,\frac{1}{2}(u^2+v^2)\, - \, u\vee v
                \, + \,\frac{1}{3}. 
\end{equation*}
We now exploit the connection to the classical Cram\'er--von Mises goodness-of-fit statistic for testing the hypothesis of uniformity on the unit interval. Let $(X_i)_{i\in\bN}$ be a sequence of
independent $\unif(0,1)$-distributed random variables.  
The test statistics may be written as $T_n=n^{-1}\sum_{i,j=1}^n h(X_i,X_j)$ with 
\begin{equation*}
           h(x,y)\,:=\,\int_0^1(1_{[0,s]}(x)-s)(1_{[0,s]}(y)-s)\,ds\, =\, \frac{1}{2}(x^2+y^2)-x\vee y +\frac{1}{3}.
\end{equation*}
Hence the off-diagonal part of the sum is a $U$-statistic with kernel $\tilde g_1$. The statement in (a) now 
follows from the general comments in Section~\ref{subsec:processes}  and the treatment of the limiting null 
distribution given e.g.\ in~\cite[Example 12.13]{vdV}.

\vspace{.7mm}
(b) We have
\begin{equation*}
g_2(u,v)\,  = \,12 \,\Tilde g_2(u,v),\quad\text{with }\,\Tilde g_2(u,v)\, = \, \frac{1}{2}(u - v)^2 \, - \,
                \frac{1}{2}(u + v)\, + \,u\wedge v + \frac{1}{12}. 
\end{equation*}
The kernel $\tilde g_2$ is connected to Watson's goodness-of-fit test for testing the 
hypothesis of uniformity on the unit circle; see \cite{Durbin} 
or \cite[Chapt.\,3]{ShorackWellner}. Let $B=(B_t)_{0\le t\le 1}$ be the Brownian bridge. 
Then a representing Gaussian process $X^{\text{\tiny Wa}}$ may be written as 
\begin{equation*}
        X^{\text{\tiny Wa}}_t =B_t-\int_0^1 B_s\, ds,
\end{equation*} 
and, again, a decomposition of the corresponding kernel is known~\cite[p.\,220]{ShorackWellner}. 
Taking the scalar factor into account we obtain the statement in (b).

\vspace{.7mm}
(c) The kernel $g_3$ can be written as
\begin{equation*}
   g_3(u,v)\,  = \frac{1}{2} g_2(u,v)\, + \,h^*(u)\,h^*(v),\,
                  \text{ with }\, h^*(u)\, :=\, \sqrt{2}\, \Bigl (\frac{1}{2} - u\Bigr ), \ 
                       0\le u\le 1.
\end{equation*}
Let $\varphi_j,\,\psi_j$, $j\in\bN$, be the functions defined in part (b). Noticing   
\begin{equation}\label{eq:Fourier2}
  \int_0^1 h^*(u)\,\varphi_j(u)\,du = 0\quad\text{and}\quad \int_0^1 h^*(u)\,\psi_j(u)\,du = \frac{1}{\pi j}
\end{equation}  
for all $j\in\bN$, we deduce that $\lambda_j = 3/(2\pi^2j^2 )$, $j\in\bN$, are eigenvalues
of $K_3$, with associated eigenfunctions $\varphi_j$, $j \in \bN$.

To derive the second set of eigenvalues
and eigenfunctions stated in the theorem, we argue as in \cite{Ba} and \cite{BaTa}, 
where related problems are dealt with. A solution 
$(\mu,\chi)$ of 
\begin{equation*}
\int_0^1 g_3(u,v)\,\chi(v)\,dv\,=\,\mu \, \chi(u),\quad  0\le u\le 1,
\end{equation*} 
with 
$\chi:[0,1]\rightarrow \bR$
orthogonal to $\{\varphi_j:j\in\bN\}$, can be expressed as
\begin{equation*}
  \chi \,=\,\sum_{k=1}^\infty \alpha_k \psi_k, \text{ with } 
                  \ \alpha_k := \langle \chi,\psi_k\rangle,
\end{equation*}
where the series converges in $\bH_1$. For $h^*$ we have the
analogous series representation 
\begin{equation*}
  h^*\,=\,\sum_{k=1}^\infty \frac{1}{\pi k}\,\psi_k.
\end{equation*}  
With $c:=\int_0^1h^*(v)\,\chi(v)\,dv$ the identity
\begin{align*}
  \int_0^1\Bigl (\frac{1}{2}\,g_2(u,v)\,+\,h^*(u)h^*(v)\Bigr )\,\chi(v)\,dv\ 
        = \ \mu\, \chi(u),\quad 0\le u\le 1,
\end{align*}
can equivalently expressed as
\begin{equation*}
  \sum_{k=1}^\infty\Bigl (\alpha_k\frac{3}{2\pi^2 k^2}\,+\,\frac{1}{\pi k}\,c\Bigr )\,\psi_k\ 
       = \ \sum_{k=1}^\infty \mu \,\alpha_k\psi_k.
\end{equation*}
From this it follows that
\begin{equation*}
  \alpha_j \Bigl (\frac{3}{2\pi^2 j^2} \, - \, \mu \Bigr ) \ = \ -\frac{c}{\pi j}\quad \text{for all }\,j\in\bN.
\end{equation*}
If there was a $j_0\in\bN$ such that $\mu = \frac{3}{2\pi^2j_0^2}$, 
then we would have that
$c=0$, and therefore that $\alpha_j=0$ for all $j\neq j_0$, i.e. $\chi = \alpha_{j_0}\psi_{j_0}$, which due
to \eqref{eq:Fourier2} is impossible.
Thus,
\begin{equation}\label{eq:eigenphi}
  \chi\, = \,-c\sum_{k=1}^\infty\frac{\frac{1}{\pi k}}{\frac{3}{2 \pi^2 k^2}\,-\mu }\,\psi_k.
\end{equation}
Using this we obtain that
\begin{align*}
  1\ = \int_0^1 \frac{\chi(u)}{c}\,h^*(u)\,du &\,=\,-\bigintsss_0^1\biggl ( \sum_{j=1}^\infty\frac{\frac{1}{\pi j}}{\frac{3}{2 \pi^2 j^2}\,-\mu }\,\psi_j(u)\biggr )\biggl (\sum_{k=1}^\infty \frac{1}{\pi k}\,\psi_k(u) \biggr )\,du\\
                                                &\,=\,- \sum_{k=1}^\infty \frac{\frac{1}{\pi^2 k^2}}{\frac{3}{2\pi^2 k^2}\,-\,\mu }\,=\, -\frac{2}{3}  \sum_{k=1}^\infty \frac{1}{1 \, - \, \frac{2\pi^2 k^2}{3} \mu }.
\end{align*}
We define the function
\begin{equation*}
  \Tilde\omega(z)\,:=\,1\,+\,\frac{2}{3} z \sum_{k=1}^\infty \frac{1}{z\,-\,\frac{2\pi^2 k^2}{3}},\,\quad z\in\bR,
\end{equation*}
and note that $\gamma = 1/\mu$ is a positive zero of $\Tilde \omega$. For $z = \frac{2\pi^2}{3} w^2$ we have
\begin{equation*}
  \Tilde\omega(z)\,=\,1\,+\,\frac{2}{3}w^2\sum_{k=1}^\infty\frac{1}{w^2\,-\,k^2}
  \ = \ 1\,+\,\frac{1}{3}\bigl (\pi w \cot(\pi w) \,-\,1\bigr ),
\end{equation*}
see \cite[Formula 1.421, 3.]{GradRyzh} for the last identity. This yields \eqref{eq:omega}. To derive
\eqref{eq:Fourier}, let $\mu_j=1/\gamma_j$ for $j\in \bN$, with $(\gamma_j)_{j\in\bN}$ the increasing
sequence of simple positive zeros of $\omega$, and $\chi_j$ for $j\in\bN$ the eigenfunction associated with
$\mu_j$.  By \eqref{eq:eigenphi} we get using the Fourier sine expansion
\cite[formula 1.445, 5.]{GradRyzh} that the orthonormal eigenfunction $\chi_j$ can be expressed as    
\begin{align*}
  \chi_j(u)&\,=\,\frac{\Tilde c_j\sqrt{2}}{\pi\mu_j}\,\sum_{k=1}^\infty\frac{k\sin(2\pi k u)}{k^2 - \frac{3}{2\pi^2 \mu_j}}
           \,=\,\frac{\Tilde c_j\sqrt{2}}{\mu_j}\,\frac{\sin \left (\sqrt{\frac{3}{2\pi^2\mu_j}}(\pi - 2\pi u)\right )}
             {2\sin \left (\sqrt{\frac{3}{2\mu_j}}\right )}\\
           &\,=\,\frac{\Tilde c_j}{\sqrt{2}\mu_j}\left (\cos\bigl (\sqrt{6/\mu_j}\,u\bigr )\ - \ \cot \bigl (\,\sqrt{3/(2\mu_j)}\,\bigr )\sin \bigl (\sqrt{6/\mu_j}\,u\bigr ) \right )
\end{align*}
for $0< u < 1$, with some real positive constant $\Tilde c_j$. From $0=\omega(1/\mu_j)$ it follows that $\cot \bigl (\,\sqrt{3/(2\mu_j)}\,\bigr )\,=\,-\sqrt{8\mu_j/3}$. Putting $c_j=\frac{\Tilde c_j\sqrt{2}}{\pi\mu_j}$ this proves \eqref{eq:Fourier}.
Verify 
\begin{equation*}
  c_j=\left (\frac{1}{2} + \frac{5\mu_j}{3}  +
    \frac{1}{4}\bigl ( 1 - \frac{8\mu_j}{3} \bigr )\sqrt{\frac{\mu_j}{6}} \sin\bigl (2\sqrt{6/\mu_j}\bigr ) -
    \frac{\mu_j}{3} \cos\bigl (2\sqrt{6/\mu_j}\bigl )\right )^{-1/2}
\end{equation*}
by simple calculation; then write $2\sqrt{6/\mu_j}=4\sqrt{3/(2\mu_j)}$, and use the trigonometric identities
\begin{equation*}
\sin 4x\,=\,4\cot x\,\frac{\cot^2 x - 1}{(\cot^2 x +1)^2}\quad\text{and}\quad
\cot 4x\,=\,\frac{\cot^4 x - 6 \cot^2 x +1}{(\cot^2 x +1)^2} 
\end{equation*}  
to obtain that $c_j=(1/2\,+\,2\mu_j)^{-1/2}$.
\end{proof}

By the trace formula \eqref{eq:trace} and $\int_0^1g_3(u,u)\,du\,=\,2/3$, we should have
$1/4\,+\, \sum_{k=1}^\infty 1/\gamma_k\,=\,2/3$. Using a numerical
calculation of the first $m=10^6$ positive zeros of $\omega$ we found that 
$1/4\,+\,  \sum_{k=1}^m 1/\gamma_k\,=\,0.66666651467749\ldots$.

\begin{remark}\label{rem:repr}
(a) The Wiener process (or Brownian motion) $W=(W_t)_{0\le t\le 1}$ is a centered Gaussian process
with covariance function $\rho_W(s,t)=s\wedge t$, and the Brownian bridge $B=(B_t)_{0\le t\le 1}$ is a centered Gaussian process with covariance function $\rho_B(s,t)=s\wedge t -st$, 
$0\le s,t\le 1$. Both may be 
regarded as random elements of $\bH_1$. Beyond the distributional level these processes can be related
pathwise, e.g.\ by measurable functions  $\Psi:\bH_1\to \bH_1$. Indeed, we may define $B$ in terms of $W$
by $B_t:=W_t-tW_1$, $0\le t\le 1$. In the other direction we may write $W_t=B_t+t\xi$, 
$0\le t\le 1$, where 
$\xi$ is a standard normal random  variable, independent of $W$;  see also~\cite[p218]{Kallenberg}. 
Such arguments can be used to obtain background processes for parts (a) and (c) in the above
proposition from the process  $X^{\text{\tiny (b)}}:=\sqrt{12}X^{\text{\tiny Wa}}$ in the proof 
of part (b), via 
\begin{equation*}
      X^{\text{\tiny (a)}}_t 
               :=   \frac{1}{\sqrt{2}}X^{\text{\tiny{(b)}}}_t + \sqrt{6}\Bigl(t-\frac{1}{2}\Bigr)\xi \ \ \text{ and }\ 
     X^{\text{\tiny (c)}}_t  := \frac{1}{\sqrt{2}}X^{\text{\tiny (b)}}_t + \sqrt{2}\Bigl(t-\frac{1}{2}\Bigr)\xi,
                             \quad 0\le t\le 1,
\end{equation*}
where $\xi$ is standard normal, independent of $X^{\text{\tiny (b)}}$.
This may not be of immediate use for obtaining the spectral decompositions, but it leads to structural
results that may be of interest. The representations imply that the distributions are the same as that of
specific random shifts of a fixed scalar multiple of the Gaussian process $X^{\rm Wa}$.

\smallbreak
(b) From its Kac-Siegert representation it is easy to obtain a decomposition of a centered Gaussian 
process into a sum of independent centered Gaussian processes. In the situation of part (b) of 
Proposition~\ref{prop:eig_g1}, for example, 
we obtain $X^{\text{\tiny Wa}}=Y^{\text{\tiny Wa}}+Z^{\text{\tiny Wa}}$,  with 
\begin{equation*}
    Y^{\text{\tiny Wa}}_t := \sum_{j=1}^\infty \frac{1}{\sqrt{2}\pi j}\, \xi_j \cos(2\pi j t), \quad 
    Z^{\text{\tiny Wa}}_t := \sum_{j=1}^\infty \frac{1}{\sqrt{2}\pi j}\,\zeta_j \sin(2\pi j t),\quad 0\le t\le 1,
\end{equation*}
where $\xi_j$, $\zeta_j$ are independent standard normals.
We may regard $Y^{\text{\tiny Wa}}$ as the orthogonal projection of $X^{\text{\tiny Wa}}$ into the 
subspace $\bH_{\text{\rm\tiny e}}:=\{f\in\bH_1:\, f(x) = f(1-x)  \ \text{a.s.}\}$ of 
even functions and $Z^{\text{\tiny Wa}}$ as the projection into the space  
$\bH_{\text{\rm\tiny o}}={\bH_{\text{\rm\tiny e}}}^{\!\!\perp} 
              =\{f\in\bH_1:\, f(x) = -f(1-x) \ \text{a.s.}\}$ of odd functions. 
The decomposition $Z=X+Y$ into independent component processes, all centered Gaussians,
corresponds to the decomposition $\rho_Z=\rho_X+\rho_Y$ of the respective covariance functions. 
\brend
\end{remark}

The tensor product representation of $\bH_2$ mentioned at the end of Section~\ref{subsec:Ustat}
implies a corresponding statement for operators, see~\cite[Lemma 4.2]{NWD}, so that 
the above factorizations and the representations of the factors lead to a representation of the 
operators $K(A)$, $A\in \{B,C,D,E,F\}$, on $\bH_2$, summarized in Proposition \ref{prop:eigstruct1},
(a) - (e). Again the case $B$ has already been obtained in~\cite{NWD}.

\begin{proposition}\label{prop:eigstruct1}
Let $K(A)$ be the Hilbert--Schmidt operator on $\bH_2$ with 
kernel $h_2^A$, let $\Lambda_A$ be the set of eigenvalues of $K(A)$, $A\in\{B,C,D,E,F\}$,
and let $\bZ_*:=\bZ\setminus\{0\}$.  

\vspace{.7mm}
\noindent
\emph{(a)}
$\Lambda_B=\{\frac{6}{\pi^4j^2k^2}: j,k\in\bN\}$ is the set of eigenvalues of $K(B)$.
For $\lambda\in\Lambda_B$ let $N_{B,\lambda}=\{(j,k)\in \bN^2: \frac{6}{\pi^4j^2k^2}=\lambda\}$,
and $\Phi_{B,\lambda}=\{f_j\otimes f_k: (j,k)\in N_{B,\lambda}\}$, where
$f_j(u)=\sqrt{2}\cos(\pi j u)$, $u\in [0,1]$, for $j\in\bN$. Then 
$\# N_{B,\lambda}$ is the multiplicity of $\lambda$, and $\Phi_{B,\lambda}$ is the set of orthonormal
eigenfunctions associated with $\lambda$.

\vspace{.7mm}
\noindent
\emph{(b)} The set of eigenvalues of $K(C)$ is $\Lambda_C=\{\frac{3}{2\pi^4j^2k^2}: j,k\in\bN\}$.
For $\lambda\in\Lambda_C$ let $N_{C,\lambda}=\{(j,k)\in \bZ_*^2: \frac{3}{2\pi^4j^2k^2}=\lambda\}$,
and $\Phi_{C,\lambda}=\{f_j\otimes f_k: (j,k)\in N_{C,\lambda}\}$, where
$f_j(u)=\sqrt{2}\cos(2\pi j u)$, $u\in [0,1]$, for~$-j\in\bN$ and
$f_j(u)=\sqrt{2}\sin(2\pi j u)$, $u\in [0,1]$, for $j\in\bN$. 
Then $\# N_{C,\lambda}$ is the multiplicity of $\lambda$ and $\Phi_{C,\lambda}$ is the set of orthonormal
eigenfunctions associated with $\lambda$.

\vspace{.7mm}
\noindent
\emph{(c)}
The set of eigenvalues of $K(D)$ is $\Lambda_D=\{\frac{3}{\pi^4j^2k^2}: j,k\in\bN\}$.
For $\lambda\in\Lambda_D$ let $N_{D,\lambda}=\{(j,k)\in \bZ_*\times \bN: \frac{3}{\pi^4j^2k^2}=\lambda\}$,
and $\Phi_{D,\lambda}=\{f_j\otimes \Tilde f_k: (j,k)\in N_{D,\lambda}\}$, where
$f_j(u)=\sqrt{2}\cos(2\pi j u)$, $u\in [0,1]$, $f_{-j}(u)=\sqrt{2}\sin(2\pi j u)$, $u\in [0,1]$, for~$j\in\bN$, and
$\Tilde f_k(u)=\sqrt{2}\cos(\pi k u)$, $u\in [0,1]$, for $k\in\bN$. Then
$\# N_{D,\lambda}$ is the multiplicity of $\lambda$ and $\Phi_{D,\lambda}$ is the set of orthonormal
eigenfunctions associated with $\lambda$.

\vspace{.7mm}
\noindent
\emph{(d)}
$\Lambda_E=\{\frac{3}{\pi^4j^2k^2}: j,k\in\bN\}$ is the set of eigenvalues of $K(E)$.
For $\lambda\in\Lambda_E$ let $N_{E,\lambda}=\{(j,k)\in \bN\times \bZ_*: \frac{3}{\pi^4j^2k^2}=\lambda\}$,
and $\Phi_{E,\lambda}=\{\Tilde f_j\otimes f_k: (j,k)\in N_{E,\lambda}\}$, where
$\Tilde f_j(u)=\sqrt{2}\cos(\pi j u)$, $u\in [0,1]$, for $j\in\bN$, and
$f_k(u)=\sqrt{2}\cos(2\pi k u)$, $u\in [0,1]$, $f_{-k}(u)=\sqrt{2}\sin(\pi k u)$, $u\in [0,1]$, for $k\in\bN$.
Then $\# N_{E,\lambda}$ is the multiplicity of $\lambda$ and $\Phi_{E,\lambda}$ is the set of orthonormal
eigenfunctions associated with $\lambda$.

\vspace{.7mm}
\noindent
\emph{(e)} For $j\in \bZ_*$, let $\lambda_j=\frac{3}{2 \pi^2 j^2}$,
  $f_j=\varphi_j$ if $j\in\bN$, and $\lambda_j=1/\gamma_{-j}$, $f_j = \chi_{-j}$ if $-j\in \bN$,
  with $\varphi_k$, $\chi_k$, and $\gamma_k$ for $k\in\bN$ as defined in Proposition~\ref{prop:eig_g1} (c).
  Then $\Lambda_F=\{\frac{1}{2} \lambda_j \lambda_k: (j,k)\in\bZ_*^2\}$ is the set of eigenvalues of
  $K(F)$. For $\lambda\in \Lambda_F$ let $N_{F,\lambda}=\{(j,k)\in \bZ_*^2: \frac{1}{2} \lambda_j \lambda_k = \lambda \}$,
  and $\Phi_{F,\lambda}=\{f_j\otimes f_k: (j,k)\in N_{F,\lambda}\}$.
  Then $\# N_{F,\lambda}$ is the multiplicity of $\lambda$ and $\Phi_{F,\lambda}$ is the set of orthonormal
  eigenfunctions associated with $\lambda$.
\end{proposition}

On this basis we now obtain the limit distributions for the statistics of interest.
Again, the case $T^B$ already appears in~\cite[Theorem 4.1]{NWD}. 

\begin{theorem}\label{thm:asdistr}
Let $\eta_{j,k}$, $j,k\in\bZ_*$, be an array of independent standard normal random variables,
and let
\begin{align*}
  Z_B \;&:=\; \sum_{j,k\in\bN}\frac{36}{\pi^4 j^ 2 k^2}\,\bigl (\eta_{j,k}^2 -1\bigr),\hspace*{0.88cm}
          Z_C \;:=\; \sum_{j,k\in\bZ_*}\frac{9}{\pi^4 j^2 k^2}\bigl(\eta_{j,k}^2 -1\bigr),\\
  Z_D \;&:=\; \sum_{j\in\bZ_*,k\in\bN}\frac{18}{\pi^4 j^ 2 k^2}\,\bigl (\eta_{j,k}^2 -1\bigr), \quad 
  Z_E \;:=\; \sum_{j\in\bN,k\in\bZ_*}\frac{18}{\pi^4 j^2 k^2}\bigl(\eta_{j,k}^2 -1\bigr),\\
  Z_F \;&:=\; \sum_{j,k\in\bN}\frac{27}{4\pi^4 j^2 k^2}\,\bigl(\eta_{j,k}^2 -1\bigr)+
  \sum_{j,k\in\bN}\frac{9}{2\pi^2 j^2\gamma_{-k}}\,\bigl(\eta_{j,-k}^2 -1\bigr)\\
   &\hspace{1.2cm}+\sum_{j,k\in\bN}\frac{9}{2\pi^2k^2\gamma_{-j}}\,\bigl(\eta_{-j,k}^2 -1\bigr)
          +\sum_{j,k\in\bN}\frac{3}{\gamma_{-j}\gamma_{-k}}\,\bigl(\eta_{-j,-k}^2 -1\bigr)
\end{align*}
Then in the hypothesis case, as $n\to\infty$, 
\begin{equation*}
n\,T_n^B\todistr Z_B, \  n\,T_n^C\todistr Z_C, \  n\,T_n^D\todistr Z_D, \  n\,T_n^E\todistr Z_E, 
                                       \text{ and }   n\,T_n^F\todistr Z_F.
\end{equation*}
\end{theorem}

In contrast to the cases considered so far $K(\DE)$ is not a simple tensor product, and it takes more effort 
to determine its spectral structure.
Recall the definition of the function $h^*$ in the proof of Proposition \ref{prop:eig_g1}~(c).
With  
\begin{equation*}
  g_4(u,v)\,:=\,g_2(u,v) + 6(u - \nf{1}{2})(v-\nf{1}{2})\,=\,g_2(u,v) + 3h^*(u)h^*(v)
\end{equation*}
for $(u,v)\in [0,1]^2$, and $\Tilde h_2^{\DE} := 6  h_2^{\DE}$ we can write
\begin{align*}
  \Tilde h_2^{\DE}\big ((u_1,&v_1),(u_2,v_2)\big )\\ &= g_4(u_1,u_2)g_4(v_1,v_2) -
    36(u_1-\nf{1}{2})(v_1-\nf{1}{2})(u_2-\nf{1}{2})(v_2-\nf{1}{2}) \\
    &=g_4(u_1,u_2)g_4(v_1,v_2) - 9h^*(u_1)h^*(u_2)h^*(v_1)h^*(v_2)
\end{align*}
for $(u_1,v_1), (u_2,v_2)\in [0,1]^2$, which displays $K(\DE)$ as a tensor of rank two.
The spectral decomposition of the Hilbert--Schmidt operator 
$K_4:\bH_1\to \bH_1$ defined by $(K_4 f)(u)=\int_0^1 g_4(u,v) f(v)\, dv$, $ f\in \bH_1$, is easily obtained
by replicating the arguments given for $K_3$; see part (c) of the proof of Proposition \ref{prop:eig_g1}.
We merely state the result.
\begin{proposition}\label{prop:eig_g4}
The operator $K_4$ has the simple eigenvalues 
$\lambda_j=3/(\pi^2j^2)$, $j\in\bN$, with associated orthonormal eigenfunctions
$\varphi_j(u)=\sqrt{2}\cos(2\pi j u)$, $j\in\bN$, and the simple
eigenvalues $\mu_j=1/\gamma_j$, $j\in\bN$, where $\gamma_j\in \left (\pi^2(j-1)^2/3,\pi^2j^2/3\right )$, $j\in\bN$,
are the positive zeros of the function
\begin{equation*}
   \omega(z)\,=\,1\,+\,\sqrt{3z}\cot(\sqrt{3z}),\quad z\in\bR, 
\end{equation*}
with associated orthonormal eigenfunctions
\begin{equation*}
  \chi_j(u)\,=\,c_j\left (\cos\bigl (2\sqrt{3/\mu_j}\,u\bigr )\ + \ \sqrt{\mu_j/3}\,\sin \bigl (2\sqrt{3/\mu_j}\,u\bigr )\right ),\quad 0\le u \le 1,
\end{equation*}
where~ $c_j=\left (1/2 + \mu_j/3\right )^{-1/2}$ ~for $j\in \bN$.
\end{proposition}

\begin{remark}\label{rem:special:ident}
  (a) Recalling the definitions and notations in Remark \ref{rem:repr} we have that $K_4$ is the covariance operator
  of the Gaussian process $X^{(\text{\tiny d})}$, with
  $X_t^{\text{\tiny (d)}}:=\sqrt{12}X_t^{\text{\tiny Wa}}\,+\,\sqrt{6}\bigl (t-\frac{1}{2}\bigl ) \xi,~0\le t\le 1.$

  \smallbreak
  (b) As a consequence of Proposition \ref{prop:eig_g4}, some useful identities may be derived.
Exploiting the trace formula
\begin{equation*}
  \frac{3}{2}\,=\,\int_0^1 g_4(u,u)\,du\,=\, \sum_{j=1}^\infty \frac{3}{\pi^2 j^2}\,+\,\sum_{j=1}^\infty \mu_j
  \,=\,\frac{1}{2}\,+\,\sum_{j=1}^\infty \mu_j,
\end{equation*}
the identity $\sum_{j=1}^\infty \mu_j = 1$ is found. As the integral
$\int_0^1\int_0^1 g_4(u,v)^2\,du\,dv = \frac{11}{20}$ 
is equal to the sum of the squared eigenvalues of $K_4$, we have
\begin{equation*}
  \frac{11}{20}\,=\,\sum_{j=1}^\infty \frac{9}{\pi^4 j^4} + \sum_{j=1}^\infty \mu_j^2 \,=\,\frac{1}{10}\,+\,\sum_{j=1}^\infty \mu_j^2,
\end{equation*}  
and thus obtain $\sum_{j=1}^\infty \mu_j^2 = \frac{11}{20}$. Writing
\begin{equation*}
1=\Bigl(\sum_{j=1}^\infty\mu_j\Bigr)^2\ =\ \sum_{j=1}^\infty \mu_j^2 + 2\sum_{1\le j<k<\infty}\mu_j\mu_k
\end{equation*}
we also
get $\sum_{1\le j<k<\infty}\mu_j\mu_k = \frac{9}{40}$.
\brend
\end{remark}

The spectral structure of the Hilbert--Schmidt operator $K_0(\DE):\bH_2\rightarrow \bH_2$ defined by
\begin{equation*}
(K_0(\DE)f)(u_1,v_1)=\int_0^1\int_0^1 g_4(u_1,v_1)g_4(u_2,v_2) f(u_2,v_2) \,du_2dv_2,
\end{equation*}
$(u_1,v_1)\in [0,1]^2$, $f\in \bH_2$, is obtained as a simple consequence of Proposition \ref{prop:eig_g4}.

\begin{corollary}\label{cor:DE}
  For $j\in \bZ_*$, let $\lambda_j=\frac{3}{\pi^2 j^2}$,
  $f_j=\varphi_j$ if $j\in\bN$, and $\lambda_j=\mu_{-j}$, $f_j = \chi_{-j}$ if $-j\in \bN$,
  with $\varphi_k$, $\chi_k$, and $\gamma_k$ for $k\in\bN$ as defined in Proposition~\ref{prop:eig_g4}.
  Then the positive eigenvalues of $K_0(\DE)$ are $\lambda_j\lambda_k$, $(j,k)\in \bZ_*^2$, with 
  associated orthonormal eigenfunctions $f_j\otimes f_k$, $(j,k)\in \bZ_*^2$.
\end{corollary}

Next we aim to derive the spectral decomposition of the Hilbert--Schmidt operator $\Tilde K(\DE)$
with respective kernel $\Tilde h_2^{\DE}$. We adopt the notations of the preceding corollary.

First, observe that from $\int_0^1 (u-\nf{1}{2})f_j(u)\,du\,=\,0$ for $j\in\bN$ it follows that
the products $\lambda_j\lambda_k$, $(j,k)\in \bN^2\cup (\bN\times (-\bN)) \cup ((-\bN)\times \bN)$, are positive
eigenvalues of $\Tilde K(\DE)$, with associated orthonormal eigenfunctions
$f_j\otimes f_k$,  $(j,k)\in \bN^2\cup (\bN\times (-\bN)) \cup ((-\bN)\times \bN)$. 

Second, let $\Lambda:=\Lambda_=\cup\Lambda_<$, with $\Lambda_==\{\mu_j^2:j\in\bN\}$ and
$\Lambda_<=\{\mu_j\mu_k:1\le j<k<\infty\}$, be the set of eigenvalues associated with the orthonormal
eigenfunctions $\chi_j\otimes \chi_k$, $j,k\ge 1$, of the operator $K_0(\DE)$.
Denote by $L_\lambda$ the linear subspace of functions of the form 
\begin{equation}\label{eq:eigenDE1}
  f\,=\,\sum_{(\ell,m)\in \bN^2}\alpha_{\ell,m}\chi_\ell\otimes \chi_m\ \in \ \bH_2,
\end{equation}
that are eigenfunctions of $\Tilde K(\DE)$ associated with the positive eigenvalue $\lambda$
of $\Tilde K(\DE)$. We need to consider the
cases $\lambda\in\Lambda$ and $\lambda\notin \Lambda$. Let
\begin{equation}\label{eq:defb}
  b_j\,:=\,\int_0^1(u-\nf{1}{2})\chi_j(u)\,du,\quad j\ge 1.
\end{equation}  
It is important that the $b_j$ are non-zero; in fact, we have the following result.

\begin{lemma}\label{lem:nonzero:b} Let $c_j$ be defined as in Proposition \ref{prop:eig_g4}. Then, for $j\ge 1$,
\begin{equation*}
   b_j\,=\,-\frac{c_j}{3} \mu_j\,<\,0.
\end{equation*}
\end{lemma}
\begin{proof} Evaluating the integral in~\eqref{eq:defb} leads to
  \begin{equation*}
  b_j\,=\,\frac{c_j}{4} \left [\sin \bigl (2\,(3/\mu_j)^{\nf{1}{2}}\bigr )\Bigl ((3/\mu_j)^{-\nf{1}{2}}\,+\,
    (3/\mu_j)^{-\nf{3}{2}}\Bigr )\,-\,2\,(3/\mu_j)^{-1} \right ].
  \end{equation*}
 From this the assertion follows with~$\,\cot\bigl ((3/\mu_j)^{-\nf{1}{2}}\bigr ) = -(3/\mu_j)^{-\nf{1}{2}}$  and
  \begin{equation*}
    \sin \bigl (2\,(3/\mu_j)^{\nf{1}{2}}\bigr )\,=\,
    \frac{2 \cot\bigl ((3/\mu_j)^{-\nf{1}{2}}\bigr )}{1\,+\,\cot^2\bigl ((3/\mu_j)^{-\nf{1}{2}}\bigr )}
    \,=\,-2\,\frac{(3/\mu_j)^{-\nf{1}{2}}}{1\,+\,(3/\mu_j)^{-1}}.\qedhere
  \end{equation*}
\end{proof}  

Suppose now that $\lambda \in \Lambda$ and let $a_{j,k}$ be as in~\eqref{eq:eigenDE1}.
Define  $N_\lambda :=\{(j,k)\in\bN^2: \mu_j\mu_k=\lambda\}$ and
$N_{<,\lambda} :=\{(j,k)\in N_\lambda: j<k\}$. Then from $f\in L_\lambda$ and Lemma \ref{lem:nonzero:b}
it follows that $c:=\int_0^1\int_0^1 (u-\nf{1}{2})(v-\nf{1}{2})f(u,v)\,du dv\;=0$,
and $a_{j,k}=0$ for all $(j,k)\notin N_\lambda$, so that
\begin{equation*}
  f= \sum_{(j,k)\in N_\lambda} a_{j,k}\chi_j\otimes\chi_k\quad\text{and}\quad
  \sum_{(j,k)\in N_\lambda} a_{j,k}b_jb_k\,=\,0.
\end{equation*}
If $\lambda\in\Lambda_=\cap \Lambda_<^c$, then we would have $\#N_\lambda=1$, and $f=0$.
If $\lambda\in\Lambda_<$, then noticing the expressions for $c_j$ and $b_j$ in
Proposition \ref{prop:eig_g4} and Lemma \ref{lem:nonzero:b}, 
we have $b_jb_k=\frac{1}{9}(1/2 + \lambda/3)\lambda^2$ for $(j,k)\in N_\lambda$,
and deduce from this that $L_\lambda$ is the linear subspace  
of functions of the form $\sum_{(j,k)\in \bN_\lambda} a_{j,k}\chi_j\otimes \chi_k$, with coefficients
satisfying $\sum_{(j,k)\in N_\lambda} a_{j,k}=0$. Thus, $L_\lambda$ has dimension~ $d_\lambda=\# N_\lambda - 1$.
An orthonormal basis of $L_\lambda$ is easily obtained as follows. Pick an arbitrary pair $(j_0,k_0)\in N_\lambda$.
The set of functions $\chi_j\otimes \chi_k - \chi_{j_0}\otimes \chi_{k_0}$, $(j,k)\in N_\lambda\setminus \{(j_0,k_0)\}$,
is a basis of $L_\lambda$. Use the Gram--Schmidt orthonormalization process to get an orthonormal basis of $L_\lambda$.  

\smallbreak
Next let $\lambda \notin \Lambda$. Then from $f\in L_\lambda$
it follows that
\begin{align*}
  \sum_{(\ell,m)\in  \bN^2}a_{\ell,m}\,\mu_\ell\mu_m\,\chi_\ell(u)\chi_m(v)\; - \; &36\,c\,(u-\nf{1}{2})(v-\nf{1}{2})\\
   &= \quad \lambda  \sum_{(\ell,m)\in  \bN^2}a_{\ell,m}\chi_\ell(u)\chi_m(v), 
\end{align*}
with $c=\int_0^1\int_0^1 (u-\nf{1}{2})(v-\nf{1}{2})f(u,v)\,du dv \neq 0$, where the identity refers to $\bH_2$. Hence,
$a_{\ell,m}\,\mu_\ell\mu_m\,-\,\lambda\,=\,36\,c\,b_\ell b_m$, i.e.
\begin{equation*} 
a_{\ell,m}\,=\,\frac{36\,c\,b_\ell b_m}{\mu_\ell\mu_m \,-\,\lambda}\quad\text{for all }\,(\ell,m)\in\bN^2.
\end{equation*}  
Using this we get from the definition of $c$ that
$1\,=\,\sum_{\ell,m)\in\bN^2}\frac{36\,b_\ell^2 b_m^2}{\mu_\ell\mu_m \,-\,\lambda}$. Thus, with
\begin{equation*}
  \omega_{\DE}(z)\,:=\,1\,+\,36\,z \sum_{(\ell,m)\in\bN^2}\frac{b_\ell^2b_m^2}{1-\mu_\ell\mu_m\,z},\quad z\in \bR,
\end{equation*}
we find that $z=1/\lambda$ is a positive zero of $\omega_{\DE}$, and 
$f\,=\,\sum_{(\ell,m)\in\bN^2}\frac{b_\ell b_m}{\mu_\ell\mu_m - \lambda}\,\chi_\ell\otimes \chi_m$.
In the other direction, it is easily seen that a function $f$ associated with a zero of $\omega_{\DE}$ in this way is
an eigenfunction of $\Tilde K(\DE)$. For a transparent statement of the result     
let $z_1^*< z_2^*<\dots$ be the reciprocal elements $\lambda_i^*$, $i\in\bN$, of $\Lambda$ arranged in increasing order.
Putting
\begin{align*}
  b_i^*\ :&=\ \sum_{(\ell,m)\in\bN^2,\,\mu_\ell \mu_m=\lambda_i^*} b_\ell^2 b_m^2\mu_\ell\mu_m\\
  &=\ \frac{1}{81}\sum_{(\ell,m)\in\bN^2,\,\mu_\ell \mu_m=\lambda_i^*} \frac{\mu_\ell\mu_m}{(1/2 + \mu_\ell/3)(1/2 + \mu_m/3)}
\end{align*}  
for $i\in\bN$, we can write  
\begin{equation}\label{eq:omegaDE}
  \omega_{\DE}(z)\,=\,1\,+\,36\,z \sum_{i\in\bN}\frac{b_i^*}{z_i^*-z},\quad z\in \bR.
\end{equation}
Due to $\frac{d}{dz} \omega_{\DE}(z)=36\ \sum_{i\in\bN}\frac{b_i^*z_i^*}{(z_i^*-z)^2}>0$, $z\in \bR$, $z\neq z_i^*$ for
$i\in \bN$, it follows that $\omega_{\DE}$ is positive and strictly increasing in the interval $(0,z_1^*)$,
continuous, strictly increasing in $(z_i^*,z_{i+1}^*)$, with $\lim_{z\downarrow z_i^*}\omega_{\DE}(z)=-\infty$,
$\lim_{z\uparrow z_{i+1}^*}\omega_{\DE}(z)=\infty$, for $i\in\bN$, and therefore has the pairwise distinct zeros
$\Hat z_i\in (z_i^*,z_{i+1}^*)$, $i\in\bN$. Putting 
\begin{equation*}
  \Hat \lambda_i=1/{\Hat z_i},\quad\text{and}\quad \Hat f_i\,=\,c_i\sum_{(\ell,m)\in\bN^2}\frac{b_\ell b_m}{\mu_\ell\mu_m - \Hat \lambda_i}\,\chi_\ell\otimes \chi_m\quad\text{for }\,i\in\bN, 
\end{equation*}
with normalizing constants $c_i\neq 0$ such that $\|\Hat f_i\|=1$, we have that
the $\Hat f_i$, $i\in\bN$, are the orthonormal eigenfunctions of $\Tilde K(\DE)$ of the form \eqref{eq:eigenDE1}, associated with the pairwise distinct eigenvalues $\Hat \lambda_i$, $i\in \bN$.   

If $\lambda\in \Lambda_<$ is an eigenvalue with the property that $\# N_\lambda =2$, i.e. $d_\lambda=1$,
we have a unique pair $(j,k)\in N_{<,\lambda}$ such that $\mu_j\mu_k=\lambda$. Then~
$\tau_{\lambda,1}\,=\,\pm 2^{-1/2}(\chi_j\otimes \chi_k\, -\, \chi_k\otimes \chi_j)$ ~are the only
normalized eigenfunctions of the form \eqref{eq:eigenDE1} associated with the eigenvalue
$\lambda$ of $\Tilde K(\DE)$.
In any case, the $\mu_j\mu_k$, $1\le j<k<\infty$, are eigenvalues of $\Tilde K(\DE)$, with associated
eigenfunctions of the form \eqref{eq:eigenDE1}. Exploiting the trace formula for $\Tilde K(\DE)$ and using
the identities given in Remark \ref{rem:special:ident} we obtain
\begin{align*}
2\,=\,\int_0^1\int_0^1 \tilde g_2^{DE}\left ((u,v),(u,v)\right )\,du\,dv &\,=\,\sum_{j,k=1}^\infty\frac{9}{\pi^4 j^2k^2}
                                                                             \,+\,2\sum_{j,k=1}^\infty\frac{3}{\pi^2 j^2}\,\mu_k\\
  &\hspace{4mm}\,+\,\sum_{\lambda\in\Lambda_<} d_\lambda \lambda\,+\,\sum_{i\in \bN} \Hat \lambda_i,
\end{align*}
i.e. $s:=\sum_{i\in \bN} \Hat \lambda_i\,=\,3/4\,-\,\sum_{\lambda\in\Lambda_<} d_\lambda \lambda$. Thus, 
if it is true that the products $\mu_j\mu_k$, $1\le j\le k<\infty$, are pairwise different,
then $s\,=\,3/4\,+\,\sum_{1\le j<k<\infty}\mu_j\mu_k\,=\,21/40\,=\,0.525$. 

Some numerical work may be done to get for suitably chosen $m\in \bN$ approximations of the $m$ largest
eigenvalues $\Hat \lambda_1>\ldots>\Hat \lambda_m$. Choosing some integer $m_1$ such that $m_1(m_1+1)/2>m$, and taking
the function $\bar \omega_{\DE}(z):=1+36\,z \sum_{j,k=1}^{m_1}\frac{b_\ell^2b_m^2}{1\,-\,\mu_\ell\mu_m z},~z\in \bR$,
as approximation of $\omega_{\DE}$, one may regard the reciprocals
of the  $m$ smallest zeros $0<\bar z_1<\ldots<\bar z_m$ of $\bar \omega_{DE}$ 
as approximations of the eigenvalues $\Hat \lambda_i$, $i=1,\ldots,m$. Proceeding in this way, with
$m=30000$ and $m_1=5000$ we get pairwise distinct numerical values of the $m_1(m_1+1)/2$    
products $\mu_j\mu_k$, $1\le j\le k\le m_1$, with the property that the distance between any two
different values is greater than 0.149, and arrive at
$\bar s=\sum_{i=1}^m \frac{1}{\bar z_i} = 0.5248875879199677$ as approximation for $s$.

Proposition \ref{prop:DE} summarizes the main result obtained for the Hilbert--Schmidt operator $K(\DE)$.   

\begin{proposition}\label{prop:DE}
  For $\lambda\in \Lambda_<$, let $\tau_{\lambda,1},\ldots,\tau_{\lambda,d_\lambda}$ be an orthonormal basis of $L_\lambda$.
  Then, a complete system of orthonormal eigenfunctions of
  $K(\DE)$, with associated positive eigenvalues,
  is given by 
  \begin{equation}\label{eq:EW:DE}
    \begin{aligned}
  &\varphi_j\otimes \varphi_k,~~\text{with}\hspace{2.5mm}\frac{3}{2\pi^4j^2k^2},~(j,k)\in\bN^2,\\
  &\varphi_j\otimes \chi_k,~~\text{with}\hspace{2.5mm}\frac{1}{2\pi^2j^2}\mu_k,~(j,k)\in\bN^2,\\
  &\chi_j\otimes \varphi_k,~~\text{with}\hspace{2.5mm}\frac{1}{2\pi^2k^2}\mu_j,~(j,k)\in\bN^2,\\
  &\tau_{\lambda,j},~~\text{with}\hspace{2.5mm}\frac{1}{6}\lambda,~j=1,\ldots,d_\lambda,\,\lambda\in \Lambda_<,\\
  &\Hat f_i,~~\text{with}\hspace{2.5mm}\frac{1}{6}\Hat \lambda_i,~i\in\bN.
    \end{aligned}
\end{equation}
\end{proposition}

The limiting null distribution of the scaled statistics~$n T_n^{\DE}$, $n\in\bN$, is now obtained as a simple 
consequence of Proposition \ref{prop:DE}. 

\begin{theorem}\label{thm:asdistrDE}
  Let $\eta_{i,j,k}$, $i \in \{1,2,3\}$, $j,k\in\bN$, $\zeta_{\lambda,i}$, $i\in \{1,\ldots,d_\lambda\}$,
  $\lambda\in \Lambda_<$, and   $\xi_i$, $i\in\bN$, be independent arrays of independent standard normal random variables,
and let
\begin{align*}
   Z_{\DE} \; :=\; &\sum_{j,k\in\bN}\frac{9}{\pi^4 j^ 2 k^2}\,\bigl (\eta_{1,j,k}^2 -1\bigr)\,+\,
              \sum_{j,k\in\bN}\frac{3}{\pi^2j^2}\,\mu_k\,\bigl(\eta_{2,j,k}^2 -1\bigr)\\
              &\,+\,\sum_{j,k\in\bN}\frac{3}{\pi^2k^2}\,\mu_j\,\bigl(\eta_{3,j,k}^2 -1\bigr)
            +\sum_{\lambda\in \Lambda_<}\lambda\, \sum_{i=1}^{d_\lambda}\,\bigl (\zeta_{\lambda,i}^2 -1 \bigr )\,+\,
              \sum_{i\in\bN} \Hat \lambda_i\,\bigl (\xi_i^2 -1\bigr )
\end{align*}
Then in the hypothesis case, as $n\to\infty$, 
\begin{equation*}
n\,T_n^{\DE}\todistr Z_{\DE}.
\end{equation*}
\end{theorem}

If it is true that the products $\mu_j\mu_k$, $1\le j\le k<\infty$, are pairwise different, then the term in the
fourth line of \eqref{eq:EW:DE} would read 
\begin{equation*}\label{eq:special:expression}
    2^{-1/2}(\chi_j\otimes \chi_k\, -\, \chi_k\otimes \chi_j)~~\text{with}\hspace{2.5mm}\frac{1}{6}\mu_j\mu_k,\,1\le j<k<\infty,
\end{equation*}  
and the fourth summand of $Z_{\DE}$ could be replaced by~ $\sum_{1\le j<k<\infty}\mu_j\mu_k \bigl (\Tilde \zeta_{j,k}^2 -1 \bigr )$, ~with $\Tilde \zeta_{j,k}$, $1\le j<k<\infty$, an array of independent standard normal variables that is
independent of $(\eta,\xi)$, where  $\eta$ and $\xi$ stand for the arrays of $\eta$ and $\xi$ variables
stated in the theorem. 

It is easy to see that the limiting distribution functions appearing in Theorem \ref{thm:asdistr}
and Theorem \ref{thm:asdistrDE} are continuous and strictly increasing.
In particular, for a given significance level $\alpha\in (0,1)$, the tests defined at the beginning
of this section are asymptotically equivalent to the tests that reject the hypothesis $H_0$ if $nT^A_n$ exceeds
the upper $\alpha$-quantile of the corresponding limit distribution; also, they are consistent. 

\begin{remark}\label{rem:repr2}
Partitions of the index set $\bZ_*\times \bZ_*$ in Theorem~\ref{thm:asdistr}
can be used to obtain representations  of the $Z$-variables as sums of independent random
variables. A decomposition into quadrants, and
\begin{align*}
   Y\ &:= \ \sum_{j,k\in\bN}\frac{9}{\pi^4 j^2 k^2}\,\bigl (\eta_{1,j,k}^2 -1\bigr), \\
   V\ &:=\  \sum_{j,k\in\bN}\frac{9}{2\pi^2 j^ 2 \gamma_k}\,\bigl (\eta_{2,j,k}^2 -1\bigr), \\ 
   W\ &:=\ \sum_{j,k\in\bN}\frac{3}{\gamma_j\gamma_k}\,\bigl (\eta_{3,j,k}^2 -1\bigr),
\end{align*}
with independent standard normal $\eta$-variables, lead to   $Z_B\eqdistr 4Y_1$, $Z_C\eqdistr Y_1+Y_2+Y_3+Y_4$, 
$Z_D\eqdistr 2(Y_1 + Y_2) \eqdistr Z_E$ and $Z_F \eqdistr 3\,Y_1+V_1+V_2+W$,
where $Y_1,\ldots,Y_4,V_1,V_2$ and $W$ are independent, $Y_1,\ldots,Y_4$  have the same distribution as $Y$,
and $V_1,V_2$ have the same distribution as $V$. Theorem~\ref{thm:asdistrDE} can similarly be used to obtain
a decomposition of $Z_{\DE}$.
\brend  
\end{remark}

The distributional equalities above refer to the individual random variables, not to their joint distributions.
For the latter we can use some fixed orthonormal basis of $\bH_2$ simultaneously for the individual tests; 
see Section~\ref{subsec:processes}.  This leads to a representation of the distribution 
of linear combinations as a series in polynomials of degree 2 in terms of independent 
standard normals, now including mixed (non-quadratic) terms; see~\cite[Theorem 12.10]{vdV}.  
The test based on $\rho^*$ in \eqref{eq:defastar} is a case in point. Further, 
such a simultaneous representation with a fixed basis provides an approach for general convex combinations
to the asymptotic null distributions of consistent tests that are based on quasirandomness.

\section{Efficiencies}\label{sec:eff}

We now use various efficiency concepts to compare the tests $T^A$, with $A$ an element of  $\{B,C,D,E,F,\DE\}$. Our main results in this section are Theorem~\ref{final:loc:ex:Ba} and 
equation~\eqref{eq:effsequal}. For the first of these we make use of a large deviation result
for $U$-statistics obtained by Nikitin and Ponikarov~\cite{NikiPoni}. Other efficiency concepts
are discussed at the beginning of Section~\ref{subsec:alteff}. 
For background, technical details, proofs and additional material we refer the reader to 
Chapters 14 and 15 in~\cite{vdV} and to the monograph~\cite{Nikitin}. 


\subsection{Bahadur efficiency}\label{subsec:baha}

Let $\{P_\theta:\theta\in\Theta\}$, $\Theta \subset \bR$ a non-empty open interval, be an injectively
parameterized family of probability distributions on some measurable space $(E,\cF)$. We assume that
$Z_1,Z_2,\ldots$ are independent and identically distributed random variables with values 
in $E$ and distribution $P_\theta$, with some (unknown) $\theta\in\Theta$, and we consider 
tests of the  hypothesis $H_0:\theta=\theta_0$ against the (two-sided) alternative 
$H_1:\theta\in\Theta_1$, $\Theta_1:=\Theta\setminus\{\theta_0\}$. 
Let $T_i=(T_{i,n})_{n\in\bN}$, $i=1,2$, be two sequences of test statistics, 
where $T_{i,n}=T_{i,n}(Z_1,\ldots,Z_n)$ is a function of the first $n$ variables.
We exclusively deal with tests rejecting the hypothesis for large values of the test statistics.
For given significance level $\alpha\in (0,1)$ and sample size $n$ let 
$L_{T_i,n}=P_{\theta_0}(T_{i,n}\ge t)_{t=T_{i,n}}$ be the level attained by $T_{i,n}$. 
We consider the nonrandomized level $\alpha$ test with critical region $\{L_{T_i,n}\le \alpha\}$.
Defining $t_{i,\alpha}:=\inf\left \{t \in \bR:~P_{\theta_0}\big (T_{i,n}\ge t \big )\le \alpha\right \}$ it
is easily seen that depending on whether $P_{\theta_0}\big (T_{i,n}\ge t_{i,\alpha} \big )$ is less than or
equal to $\alpha$ or bigger than $\alpha$, the critical region of the $T_{i,n}$ test is equal to
$\{T_{i,n}\ge t_{i,\alpha}\}$ or $\{T_{i,n} > t_{i,\alpha}\}$. Throughout we assume that the tests are
consistent for each given level $\alpha$, i.e. 
$\lim_{n\to\infty}\bP_\theta\big (L_{T_i,n}\le \alpha \big )=1$ for all $\theta\neq\theta_0$, and that the distributions
$P_\theta$, $\theta\in \Theta$, are absolutely continuous with respect to $P_{\theta_0}$. 
For $\beta\in(\alpha,1)$ and $\theta\in\Theta_1$ let 
\begin{equation*}
  N_{T_i}(\alpha,\beta,\theta)\,
            :=\,\inf\bigl\{m\in \bN:~P_\theta (L_{T_i,n}\le \alpha )\ge \beta\; \text{ for all }n\ge m \bigr\}.
\end{equation*}
These quantities grow to infinity as $\alpha$ tends to 0; see \cite[Proof of Theorem 14.22]{vdV}.
Let
\begin{equation*}
  {\rm eff}_{T_2,T_1}^{\bam}(\beta,\theta)\,:=\,\liminf_{\alpha\downarrow 0}\, \frac{N_{T_1}(\alpha,\beta,\theta)}{N_{T_2}(\alpha,\beta,\theta)}\  \text{ and }\ 
  {\rm eff}_{T_2,T_1}^{\bap}(\beta,\theta)\,:=\,\limsup_{\alpha\downarrow 0}\, \frac{N_{T_1}(\alpha,\beta,\theta)}{N_{T_2}(\alpha,\beta,\theta)}.
\end{equation*}
If $ {\rm eff}_{T_2,T_1}^{\bam}(\beta,\theta)={\rm eff}_{T_2,T_1}^{\bap}(\beta,\theta)$
then the limit
\begin{equation*}
  {\rm eff}_{T_2,T_1}^{\ba}(\beta,\theta)\,:=\,\lim_{\alpha\downarrow 0}\, \frac{N_{T_1}(\alpha,\beta,\theta)}{N_{T_2}(\alpha,\beta,\theta)}
\end{equation*} 
is called the \emph{(exact) Bahadur asymptotic relative efficiency} (ARE) of $T_2$ with respect 
to $T_1$. 
The Bahadur ARE may not exist, or it may be difficult to calculate. If 
the limits $\lim_{\theta\to\theta_0}{\rm eff}_{T_2,T_1}^{\bam}(\beta,\theta)$ and
$\lim_{\theta\to\theta_0}{\rm eff}_{T_2,T_1}^{\bap}(\beta,\theta)$ both exist and are equal, then
\begin{equation*}
       {\rm eff}_{T_2,T_1}^{\ba}\, = \, \lim_{\theta\to\theta_0}{\rm eff}_{T_2,T_1}^{\bam}(\beta,\theta) 
            \ = \ \lim_{\theta\to\theta_0}{\rm eff}_{T_2,T_1}^{\bap}(\beta,\theta)
\end{equation*}
is called the {\em (exact) local Bahadur {\rm ARE}} of the $T_2$ test with respect to the $T_1$ test.

Suppose that  for each $\theta\in\Theta_1$ there exist constants $0<c_{T_i}(\theta)<\infty$ such that
\begin{equation}\label{eq:slopes}
      -\frac{1}{n} \log L_{T_i,n}\ \to\ \frac{1}{2} c_{T_i}(\theta)\quad\text{in } P_\theta\text{-probability}, 
\end{equation}  
$c_{T_i}(\theta)$ is known as the {\em Bahadur slope} of $T_i$. Then the Bahadur {\rm ARE} of
$T_2$ with respect to $T_1$ at the parameter $\theta$ exists, does not depend on $\beta$,
and is given by the ratio of the respective slopes,  
\begin{equation*}
   {\rm eff}_{T_2,T_1}^{\ba}(\theta)\,=\,\frac{c_{T_2}(\theta)}{c_{T_1}(\theta)}.
\end{equation*}  
We require sufficient conditions for~\eqref{eq:slopes}.  
Suppose that there are functions $b_i:\Theta\to\bR$ such that for all $\theta\in\Theta_1$
\begin{equation*}
  T_{i,n}\to b_i(\theta)\quad\text{in } P_\theta\text{-probability}.
\end{equation*}
Additionally, let $I_i$ be open intervals that contain the set $b_i(\Theta_1)$, and suppose that
$f_i:I_i\to\bR$, $i=1,2$, are continuous functions such that
\begin{equation*}
  -\frac{1}{n} \log P_{\theta_0}(T_{i,n}\ge t)\ \longrightarrow\ \frac{1}{2}f_i(t)\quad\text{for all }t\in I_i.
\end{equation*}
Then \eqref{eq:slopes} holds, with $c_{T_i}(\theta)=f_i\bigl (b_i(\theta)\bigr )$ for 
$\theta\in \Theta_1$; see~\cite[Section 14.4]{vdV}.

Finally, given a kernel $h$ of order $k$ let $V=(V_n)_{n\ge k}$, with
\begin{equation*}
  V_n=V_n(Z_1,\ldots,Z_n)=\frac{1}{n^k}\sum_{(i_1,\ldots,i_k)\in [n]^k} h(Z_{i_1},\ldots,Z_{i_k}),\quad n\ge k,
\end{equation*}
be the sequence of associated  $V$-statistics. Our main tool in this section is the following result 
for large deviation probabilities of degenerate $U$- and $V$-statistics~\cite[Theorem 3.1]{NikiPoni}. 

\begin{theorem*}[Nikitin and Ponikarov 2001]
Let $(E,\mathcal F)$ be a Borel space. Suppose that the kernel~$h$ is bounded, of degree $k$, and
degenerate with rank 2. Suppose further  that the Hilbert--Schmidt operator $K$ associated with 
the kernel $h_2$ is positive and has 
a simple largest eigenvalue $\lambda_1>0$. Then there exist a real number $\delta>0$ and an 
analytic function $d:(-\delta,\delta )\rightarrow \bR$, represented by the power series
\begin{equation*}
d(s)\,=\,\frac{1}{\lambda_1\binom{k}{2}}s^2 + \sum_{j=3}^\infty d_js^j,\quad s\in (-\delta,\delta),
\end{equation*}  
such that for all $t\in (0,\delta^2)$ and all real sequences $t_n\rightarrow 0$ it holds that 
\begin{equation*}
   \lim_{n\to\infty} -\frac{1}{n} \log P(V_n \ge t + t_n)\,
                     =\,\lim_{n\to\infty} - \frac{1}{n} \log P(U_n \ge t + t_n)\,
                     =\,\frac{1}{2}\,d\bigl (t^{1/2}\bigr ).
\end{equation*}
\end{theorem*}

\medbreak
\begin{remark}\label{rem:Borel_Space} (a) The original version of the preceding theorem deals with the
special Borel space $([0,1],\cF_1)$. For a general Borel space $(E,\cF)$ we have a Borel set
$R\subset \bR$ such that $(E,\cF)$ and the measurable space $(R,\cS)$, with $\cS$
the $\sigma$-field of Borel sets of $R$, are Borel isomorphic, i.e.\ there is a bijective map
$\Psi:R\rightarrow E$ with the property that $\Psi$ and its inverse $\Psi^{-1}$ are measurable.
Using this it is easily verified that the assertion of the theorem carries over to general Borel spaces.  

(b) It is also not explicitly requested in the original version of the theorem that the
Hilbert--Schmidt operator $K$ associated with
the kernel $h_2$ is positive, but this is needed 
in connection with the `normalization conditions' stated in the proof of the theorem.
\brend
\end{remark}

We now assume that $\theta_0=0\in \Theta$, and that $\{P_\theta:\theta\in\Theta\}$
is a family of probability distributions on $E=[0,1]^2$ such that,
for each $\theta\in \Theta$, the distribution function of $P_\theta$ is a bivariate copula $C_\theta$, with
$C_0=C\ind$ the independence copula. To derive the Bahadur slope of the $T^A$ tests we first note that, 
for all $\theta\in\Theta$,
\begin{equation*}
  T_{n}^A\rightarrow b_A(\theta)\quad\text{in }\, P_\theta\text{-probability, 
                                        with }\,b_A(\theta):=\bE_\theta\bigl(h^A(Z_1,Z_2,Z_3,Z_4)\bigr).
\end{equation*}
We impose the additional condition that for each $\theta\in\Theta$ the distribution $P_\theta$ has a
$P_0$-density $p_\theta$ that can be written as 
\begin{equation}\label{eq:localt1}
  p_\theta = 1+\theta q_\theta,\quad\text{with }\, q_\theta,q \in \bH_2,\,q \neq 0,\,\text{ such that }\,\|q_\theta - q \|\to 0
  \,\text{ as }\,\theta\to 0.
\end{equation}  
In particular, $P_\theta$ is absolutely continuous with respect to $P_0$.
With $P_0^m$ the $m$-fold product measure of $P_0$ we then obtain
\begin{equation}\label{eq:Ld:TA:H1}
\begin{aligned}
  b_A(\theta)&\,=\ \int h^A(z_1,z_2,z_3,z_4)\prod_{j=1}^4 \bigl (1+\theta q_\theta(z_j)\bigr )\,dP_0^4(z_1,z_2,z_3,z_4)\\
  &=\ 6\theta^2 \int h_2^A(z_1,z_2)\,q (z_1)\,q (z_2)\,dP_0^2(z_1,z_2)\,+\,o(\theta^2)\\
  &=\ 6\theta^2\,\langle K(A)q ,q \rangle \,+\,o(\theta^2)\quad\text{as }\,\theta\to 0.
\end{aligned}
\end{equation}                   
Note that \eqref{eq:Ld:TA:H1} implies that $b_A(\theta)>0$ if $|\theta|$ is positive and small enough.

In the hypothesis case the theorem of Nikitin and Ponikarov    
directly applies to the sequences $T^B=(T_n^B)_{n\in\bN}$, $T^F=(T_n^F)_{n\in\bN}$, and $T^{\DE}=(T_n^{\DE})_{n\in\bN}$ of
$U$-statistics, because according to Proposition \ref{prop:eigstruct1} and Proposition \ref{prop:DE}
the associated Hilbert--Schmidt operators $K(B)$, $K(F)$, and $K(\DE)$ have the respective simple largest
eigenvalues $\lambda_{B,1}=\frac{6}{\pi^4}$, $\lambda_{F,1}=\frac{1}{2\gamma_1^2}$, and
$\lambda_{\DE,1}=\Hat \lambda_1=1/\Hat z_1$, with
$\gamma_1\in (0,2\pi^2/3)$ the smallest positive zero of the function $\omega$ defined in \eqref{eq:omega},
and $\Hat z_1\in (z_1^*,z_2^*)$ the smallest positive zero of the function $\omega_{\DE}$ given in \eqref{eq:omegaDE}.
Thus, for $A\in\{B,F,\DE\}$ there exist a $\delta_A>0$ and an analytic function
$d_A$ represented in the interval $(-\delta_A,\delta_A)$ by the power series
$d_A(s)\,=\,\frac{1}{6\lambda_{A,1}}s^2 + r_A(s)$, with remainder $r_A(s)=o(s^3)$ as $s\to 0$,
such that for all $t\in (0,\delta_A^2)$ and all real sequences
$t_n\rightarrow 0$ it holds that
\begin{equation*} 
    \lim_{n\to\infty} - \frac{1}{n} \log P_0(T_n^A \ge t + t_n)\,=\,\frac{1}{2}\,d_A\bigr (t^{\nf{1}{2}}\bigl )\,>\,0.
\end{equation*}
Using this and \eqref{eq:Ld:TA:H1} we obtain that for $\theta\in\Theta_1$ such that
$0< b_A(\theta)<\delta_A^2$
the Bahadur slope $c_A(\theta)$ of $T^A$ is given by  
\begin{equation}\label{def:slope:BF}
  c_{A}(\theta)\,=\,d_A\left (b_A(\theta)^{\nf{1}{2}}\right )\,>\,0,
\end{equation}  
with 
\begin{equation*}
    c_{A}(\theta)\,=\,\frac{1}{\lambda_{A,1}}\theta^2 \int h_2^A(z_1,z_2)\,q (z_1)\,q (z_2)\,dP_0^2(z_1,z_2)\,+\,o(\theta^2)\quad\text{as }\,\theta\to 0.
\end{equation*}
Arguing as in \cite[Proof of Theorem 14.22]{vdV} we now arrive at the following result. 
\begin{proposition}\label{prop:slope:BF}
For $A\in \{B,F,\DE\}$ let $\theta_A>0$ be such that $0<b_A(\theta)<\delta_A^2$ for all $\theta\in \Theta_1$,
$|\theta|<\theta_A$. Then
\begin{equation*}
   \lim_{\alpha\downarrow 0}\,\frac{1}{N_{T^A}(\alpha,\beta,\theta)}\,\log \alpha\,=\,c_{A}(\theta)\,>\,0
                           \quad\text{for all }\,\theta\in\Theta_1,\,|\theta|<\theta_A.
\end{equation*}
\end{proposition}

To obtain a related result for the other $T^A$ tests we require a generalization of the
Nikitin--Ponikarov theorem to situations where the largest eigenvalue $\lambda_1$ has multiplicity
$m_1\ge 2$. While we may assume that $m_1=2$ in connection with the $T^A$ tests, we treat the
general case as it might be of interest in other situations.

Let $(\lambda_j)_{j\ge 1}$ be
the non-increasing sequence of positive eigenvalues of the Hilbert--Schmidt operator $K$, counted according to their
multiplicities, with $(\varphi_j)_{j\ge 1}$ the sequence of associated eigenfunctions  
of $K$. As the kernel $h$ is bounded, we may assume that the $\varphi_i$ are are bounded as well.
Let $s:=\sup\{\lambda_j:j>m_1\}$, with $s:=0$ if a positive eigenvalue of $K$ different from $\lambda_1$ does not exist.
For $\epsilon \in (0,\lambda_1-s)$ let the kernel $h_{+,\epsilon}:E^k\rightarrow \bR$ be defined by
$h_{+,\epsilon}(z_1,\ldots,z_k):=h(z_1,\ldots,z_k) + \epsilon\sum_{1\le i<j\le k}\varphi_1(z_i)\varphi_1(z_j)$ for
$(z_1,\ldots,z_k)\in E^k$. Obviously, $h_{+,\epsilon,1}=0$ $P$-almost surely, and
\begin{equation*}
  h_{+,\epsilon,2}(z_1,z_2)\,=\,h_2(z_1,z_2)\,+\,\epsilon\,\varphi_1(z_1)\varphi_1(z_2)\quad\text{for all }\,z_1,z_2\in E.
\end{equation*}  
The Hilbert--Schmidt operator $K_{+,\epsilon}$ associated with $h_{+,\epsilon,2}$ has the simple largest
eigenvalue $\lambda_{+,1,\epsilon}=\lambda_1+\epsilon$. The $V$-statistics with 
kernel $h_{+,\epsilon}$ are
\begin{equation*}
V_{+,\epsilon,n}\,=\,V_{n}\,+\,\epsilon\, \binom{k}{2} \left (\frac{1}{n}\sum_{j=1}^n \varphi_1(Z_j)\right )^2.
\end{equation*}
Similarly, defining the kernel $h_{-,\epsilon}:E^k\rightarrow \bR$ by
\begin{equation*}
  h_{-,\epsilon}(z_1,\ldots,z_k):=h(z_1,\ldots,z_k) - \epsilon\sum_{\ell = 2}^{m_1}\sum_{1\le i<j\le k}\varphi_\ell(z_i)\varphi_\ell(z_j)
\end{equation*}  
for $(z_1,\ldots,z_k)\in E^k$, we have that $h_{-,\epsilon,1}=0$ $P$-almost surely and that
\begin{equation*}
  h_{-,\epsilon,2}(z_1,z_2)
          \, =\, h_2(z_1,z_2) - \epsilon \sum_{\ell =2}^{m_1}\varphi_\ell(z_1)\varphi_\ell (z_2)\quad
  \text{for all }\,z_1,z_2\in E.
\end{equation*}  
The Hilbert--Schmidt operator $K_{-,\epsilon}$ associated with $h_{-,\epsilon,2}$ has the simple largest
eigenvalue $\lambda_{-,1,\epsilon}=\lambda_1$. The $V$-statistics with
kernel $h_{-,\epsilon}$ are
\begin{equation*}
  V_{-,\epsilon,n}\,=\,V_{n}\,-\,\epsilon\, \binom{k}{2} \sum_{\ell =2}^{m_1}\left (\frac{1}{n}\sum_{j=1}^n \varphi_\ell(Z_j)\right )^2.
\end{equation*}
Fix $\epsilon\in (0,\lambda_1-s)$.
As the Nikitin--Ponikarov theorem applies to the sequences of $V$-statistics $(V_{\pm,\epsilon,n})_{n\ge k}$, we obtain as
a first consequence that there exist a real number $\delta_\epsilon>0$ and analytic functions
$d_{\pm,\epsilon}$ represented in the interval $(-\delta_\epsilon,\delta_\epsilon)$ by power series 
$d_{\pm,\epsilon}(s)\,=\,\frac{s^2}{\lambda_{\pm,1,\epsilon}\binom{k}{2}}\,+\,r_{\pm,\epsilon}(s)$, 
with remainders $r_{\pm,\epsilon}(s)=o(s^3)$ as $s\to 0$.   
Thus, for $t\in (0,\delta_\epsilon^2)$,    
\begin{equation}\label{eq:V-statistic:+}
  \liminf_{n\to\infty}\,-\frac{1}{n} \log P_0(V_n\ge t)\,\ge\, \lim_{n\to\infty} - \frac{1}{n} \log P_0(V_{+,\epsilon,n}\ge t)\, =\, d_{+,\epsilon}(t^{1/2}),  
\end{equation}
and
\begin{equation}\label{eq:V-statistic:-}
  \limsup_{n\to\infty}\, -\frac{1}{n} \log P_0(V_n\ge t)\,\le\, \lim_{n\to\infty} -\frac{1}{n}\log P_0(V_{-,\epsilon,n}\ge t)\, =\, d_{-,\epsilon}(t^{1/2}).  
\end{equation}
The following lemma provides statements corresponding to \eqref{eq:V-statistic:+} and
\eqref{eq:V-statistic:-} for the sequence $(U_n)_{n\ge k}$ of $U$-statistics.
\begin{lemma}\label{lem:U-statistic:ungl}
  Let $\|h\|_\infty:=\sup_{(z_1,\ldots,z_k) \in E^k}|h(z_1,\ldots,z_k)|$,
  \begin{equation*}
  \tau_{-,\epsilon}:= \min\bigl (1,\delta_\epsilon^2,(3\|h\|_\infty)^{\nf{1}{2}}\bigr )~\text{\,and\,}~
  \tau_{+,\epsilon}:= \min\bigl ((3\|h\|_\infty)^{\nf{1}{2}},-\nf{1}{2} + (\delta_\epsilon^2+\nf{1}{4})^{\nf{1}{2}}\bigr )\bigr ).
  \end{equation*}
  Then, for all $t\in \bigl (0,\tau_{+,\epsilon})$,   
\begin{equation*}
  \limsup_{n\to\infty}\,-\frac{1}{n}\log P_0(U_n\ge t)\,\le\,d_{-,\epsilon}\bigl ((t(1+t))^{\nf{1}{2}}\bigr ),  
\end{equation*}
and for all $t\in \bigl (0,\tau_{-,\epsilon}\bigr )$,
\begin{equation*}
  \liminf_{n\to\infty}\,-\frac{1}{n}\log P_0(U_n\ge t)\,\ge\,d_{+,\epsilon}\bigl ((t(1-t))^{\nf{1}{2}}\bigr ).  
\end{equation*}
\end{lemma}
\begin{proof}  
  Arguing as in \cite[Subsection 2.4]{vdV} we see
  that for $t>0$, $0<\eta<\|h\|_\infty$, and for $n$ large enough,
  \begin{equation*}P_0(U_n \ge t)\,\le\, P_0(V_n \ge t - 3\eta)~\text{\,and\,}~
    P_0(U_n \ge t)\,\ge\, P_0(V_n \ge t + 3\eta).
  \end{equation*}
  With $t\in \bigl (0,\tau_{+,\epsilon}\bigr )$ and $\eta=\frac{1}{3}t^2$ it holds that $\eta<\|h\|_\infty$
  and that $t+3\eta<\delta_\epsilon^2$; similarly, 
with  $t\in \bigl (0,\tau_{-,\epsilon}\bigr )$ and $\eta=\frac{1}{3}t^2$ it holds that $\eta<\|h\|_\infty$
  and that $t-3\eta<\delta_\epsilon^2$. Using this, the assertions of the lemma follow from \eqref{eq:V-statistic:+}
  and \eqref{eq:V-statistic:-}.
\end{proof}
As $\|h^A\|_\infty\le 1$, Lemma \ref{lem:U-statistic:ungl} applies to the sequences of $U$-statistics $T^A$,
$A\in\{C,D,E\}$. To be precise, putting 
$U=T^A$ we have constants $s:=s_A$, $m_1=m_{A,1}$, $\delta_\epsilon=\delta_{A,\epsilon}>0$, $\lambda_1=\lambda_{A,1}$,
$\lambda_{\pm,1,\epsilon}=\lambda_{A,\pm,1,\epsilon}$, $\tau_{\pm,\epsilon}=\tau_{A,\pm,\epsilon}>0$, and analytic functions
$d_{\pm,\epsilon}=d_{A,\pm,\epsilon}$ such that
\begin{equation}\label{eq:U-statistic:ungl+}
  \limsup_{n\to\infty}\,-\frac{1}{n}\log P_0(T_n^A\ge t)\,\le\,d_{A,-,\epsilon}\bigl ((t(1+t))^{\nf{1}{2}}\bigr )
  \quad\text{for all }\,t\in \bigl (0,\tau_{A,+,\epsilon}\bigr ),
\end{equation}
and 
\begin{equation}\label{eq:U-statistic:ungl-}
  \liminf_{n\to\infty}\,-\frac{1}{n}\log P_0(T_n^A\ge t)\,\ge\,d_{A,+,\epsilon}\bigl ((t(1-t))^{\nf{1}{2}}\bigr )\,>\,0
  \quad\text{for all }\,t\in \bigl (0,\tau_{A,-,\epsilon}\bigr ).
\end{equation}
In what follows, $\epsilon$ is always chosen such that
$0<\epsilon<\epsilon^*:=\min\bigl \{s_A-\lambda_{A,1}:A\in \{C,D,E\}\bigr \}$.  
For $A\in \{C,D,E\}$ let $\tau_{A,\epsilon}:=\min(\tau_{A,-,\epsilon},\tau_{A,+,\epsilon}\bigr )$, and
let $\theta_{A,\epsilon}>0$ be chosen such that
$0<b_A(\theta)<\tau_{A,\epsilon}$ for all $\theta\in\Theta_1$, $|\theta|<\theta_{A,\epsilon}$.
For $\theta\in\Theta_1$, $|\theta|<\theta_{A,\epsilon}$, put
\begin{align*}
  c_{A,+,\epsilon}(\theta)&:=d_{A,+,\epsilon}\left (\left \{b_A(\theta)\bigl (1-b_A(\theta)\bigr)\right \}^{\nf{1}{2}}\right ),\\
  c_{A,-,\epsilon}(\theta)&:=d_{A,-,\epsilon}\left (\left \{b_A(\theta)\bigl (1+b_A(\theta)\bigr)\right \}^{\nf{1}{2}}\right ),
\end{align*}
if $A\in\{C,D,E\}$. Let $\theta_\epsilon^*:=\min(\theta_{C,\epsilon},\theta_{D,\epsilon},\theta_{E,\epsilon})$.

\begin{proposition}\label{prop:slope:CDE} Let $A\in \{C,D,E\}$. Then, for all $\epsilon\in (0,\epsilon^*)$,
  $\theta\in \Theta_1$, $|\theta|<\theta_\epsilon^*$, 
  \begin{equation*}
    \limsup_{\alpha\downarrow 0}\,-\frac{1}{N_{T^A}(\alpha,\beta,\theta)}\log \alpha\,\le\,
    c_{A,-,\epsilon}(\theta),
  \end{equation*}
  and
  \begin{equation*}
  \liminf_{\alpha\downarrow 0}\,-\frac{1}{N_{T^A}(\alpha,\beta,\theta)}\log \alpha\,\ge\,
  c_{A,+,\epsilon}(\theta)\,>\,0.
\end{equation*}
\end{proposition}

\begin{proof}
Again we argue as in \cite[Proof of Theorem 14.22]{vdV}. 
Fix $\epsilon\in (0,\epsilon^*)$, and $\theta\in \Theta_1$, with
$|\theta|<\theta_\epsilon^*$. By \eqref{eq:U-statistic:ungl+} and \eqref{eq:U-statistic:ungl-} 
we have that  
\begin{equation*}
	\limsup_{n\to\infty}\,-\frac{1}{n}\log L_{T^A,n}\,\le\,c_{A,-,\epsilon}(\theta)\ \text{ and }\ 
	\liminf_{n\to\infty}\,-\frac{1}{n}\log L_{T^A,n}\,\ge\,c_{A,+,\epsilon}(\theta)
\end{equation*}
with $P_\theta$-probability 1. As
\begin{equation*}
 P_\theta(L_{T^A,n}\le \alpha)\,
        =\,P_\theta\Bigl(-\frac{1}{n}\log L_{T^A,n}\ge -\frac{1}{n}\log \alpha \Bigr)
            \begin{cases}>\beta~\text{if } n=N_{T^A}(\theta,\alpha,\beta),\\
                                       \le \beta~\text{if } n=N_{T^A}(\theta,\alpha,\beta)-1,
              \end{cases}
\end{equation*}
we deduce that
\begin{equation*}
 \limsup_{\alpha\downarrow 0}\,-\frac{1}{N_{T^A}(\alpha,\beta,\theta)}\log \alpha\,\le\,
  c_{A,-,\epsilon}(\theta)
\end{equation*}
and
\begin{equation*}
 \liminf_{\alpha\downarrow 0}\,-\frac{1}{N_{T^A}(\alpha,\beta,\theta)-1}\log \alpha\,\ge\,
 c_{A,+,\epsilon}(\theta).
\end{equation*}
Finally, recall that $\lim_{\alpha\downarrow 0} {N_{T^A}(\alpha,\beta,\theta)}=\infty$.
\end{proof}

On the basis of the above we now obtain the main result of this section. 
For $\epsilon\in (0,\epsilon^*)$, and $A\in \{B,F,\DE\}$, let $\theta_{A,\epsilon}:=\theta_A$, with $\theta_A$ as
defined as in Proposition \ref{prop:slope:BF}, and put  
\begin{equation*}  c_{A,\pm,\epsilon}(\theta):=c_A(\theta),\quad \lambda_{A,\pm,1,\epsilon}:=\lambda_{A,1},
\end{equation*}
with $c_A(\theta)$ as in \eqref{def:slope:BF}.  Moreover, for $A\in \{B,C,D,E,F,\DE\}$, with $K(A)$ the associated 
Hilbert--Schmidt operator, and $q$ as defined in~\eqref{eq:localt1} encoding the direction of the alternatives,
let 
\begin{equation}\label{eq:defkappa}
    \kappa(K(A),q ):= \frac{1}{\lambda_{A,1}}\bigl\langle K(A) q ,q \bigr\rangle.
\end{equation}
Note that this value is invariant under scale transformations $T^A\mapsto \alpha T^A$, $\alpha>0$, 
of the underlying statistic.   

\begin{theorem}\label{final:loc:ex:Ba}
Let $A,A'\in \{B,C,D,E,F,\DE\}$. Then, for all $\epsilon\in (0,\epsilon^*)$, $\theta\in \Theta_1$, $|\theta|<\min(\theta_{\epsilon}^*,\theta_B,\theta_F)$, the bounds 
\begin{equation*}
\frac{c_{A',+,\epsilon}(\theta)}{c_{A,-,\epsilon}(\theta)}\,\le\,{\rm eff}_{T^{A'},T^A}^{\bam}(\beta,\theta)\,\le \,
{\rm eff}_{T^{A'},T^A}^{\bap}(\beta,\theta)\,\le\,\frac{c_{A',-,\epsilon}(\theta)}{c_{A,+,\epsilon}(\theta)}
\end{equation*}
apply, and the local exact Bahadur efficiency of $T^{A'}$ with respect to  $T^{A}$ is given by
\begin{equation}\label{eq:locexBah}
       {\rm eff}_{T^{A'},T^A}^{\ba}\,=\,\frac{\kappa(K(A'),q )}{\kappa(K(A),q )}. 
\end{equation}
\end{theorem}

\begin{proof}
The first statement is a consequence of Proposition \ref{prop:slope:BF} 
and Proposition \ref{prop:slope:CDE}. Using \eqref{eq:Ld:TA:H1} and the power series 
expansion of the analytic functions $d_{A,\pm,\epsilon}$ we get that
\begin{equation*}
 \limsup_{\theta\to 0} {\rm eff}_{T^{A'},T^A}^{\bap}(\beta,\theta)\,
                \le\,\lim_{\theta\to 0}\frac{c_{A',-,\epsilon}(\theta)}{c_{A,+,\epsilon}(\theta)}\,
                 =\,\frac{\lambda_{A,+,1,\epsilon}\langle K(A')q ,q \rangle}
                                                {\lambda_{A',-,1,\epsilon}\langle K(A)q ,q \rangle}
\end{equation*}
and
\begin{equation*}
  \liminf_{\theta\to 0} {\rm eff}_{T^{A'},T^A}^{\bam}(\beta,\theta)\,
              \ge\,  \lim_{\theta\to 0}\frac{c_{A',+,\epsilon}(\theta)}{c_{A,-,\epsilon}(\theta)}\,
                =\, \frac{\lambda_{A,-,1,\epsilon}\langle K(A')q ,q \rangle}
                                              {\lambda_{A',+,1,\epsilon}\langle K(A)q ,q \rangle}. 
\end{equation*}
for all $\epsilon \in (0,\epsilon^*)$. Let $\epsilon\downarrow 0$ to obtain the second statement.    
\end{proof}  

\begin{example}\label{ex:FGM}
We consider the special case where $\Theta =[-1,1]$ and $\{P_\theta:\theta \in\Theta \}$ is the
Farlie--Gumbel--Morgenstern (FGM) distribution family on the unit square, with copula 
\begin{equation}\label{eq:defFGM}
     C_\theta^{\rm FGM}(x,y) \; =\; xy\, +\, \theta x y (1-x)(1-y), \quad 0\le x,y\le 1,
\end{equation}
see~\cite{Farlie,Gumbel,Morgenstern}. If $\theta=0$ then the components of a two-dimensional
random vector with this distribution function are independent. Here,
\begin{equation*}
    q_\theta=\frac{dP_\theta}{dP_0} \,=\, 1 + \theta q^{\rm FGM}, 
\end{equation*}
with $q^{\rm FGM}(x,y)=(2x-1)(2y-1)$ for $x,y \in [0,1]$. Table \ref{Tab:EffFGM} shows the local 
exact Bahadur efficiency ${\rm eff}_{T^{A},T^B}^{\ba}$ of $T^A$ with respect to $T^B$,
$A\in\{B,C,D,E,F,\DE\}$, obtained from
\begin{equation*}
  \bigl \langle K(B)q^{\rm FGM} ,q^{\rm FGM} \bigr \rangle \,=\,\frac{I_1^2}{6}\, ,
  \bigl \langle K(C)q^{\rm FGM} ,q^{\rm FGM} \bigr \rangle    \,=\,\frac{I_2^2}{6}\,,
\end{equation*}
\begin{equation*}
  \bigl \langle K(D)q^{\rm FGM} ,q^{\rm FGM} \bigr \rangle= \bigl \langle K(E)q^{\rm FGM} ,q^{\rm FGM}
  \bigr \rangle =\frac{I_1I_2}{6},\  
 \bigl \langle K(F)q^{\rm FGM} ,q^{\rm FGM} \bigr \rangle=\frac{I_3^3}{2},
\end{equation*}
\begin{equation*}  
  \bigl \langle K(\DE)q^{\rm FGM} ,q^{\rm FGM} \bigr \rangle= \bigl \langle K(D)q^{\rm FGM} ,q^{\rm FGM}
  \bigr \rangle +
  \bigl \langle K(E)q^{\rm FGM} ,q^{\rm FGM} \bigr \rangle = \frac{I_1I_2}{3},
\end{equation*}  
where the integrals   
\begin{equation*}
  I_i\,=\,\int_0^1\int_0^1(2u-1)g_i(u,v)(2v-1)\,du\,dv,\quad i=1,2,3, 
\end{equation*}  
evaluate to $I_1=\frac{1}{5}$, $I_2=\frac{1}{15}$ and $I_3=\frac{4}{45}$.
For $T^F$ and $T^{\DE}$ the respective smallest roots of the functions $\omega$ and $\omega_{DE}$
defined in \eqref{eq:omega} and \eqref{eq:omegaDE} were calculated numerically to be approximately 3.492799533553498
and 3.091933991647879.
This leads to the values shown in the first row and the last two columns of Table \ref{Tab:EffFGM}.    
\brend
\end{example}

\begin{table}    
\begin{tabular}{c C{1cm}C{1cm}C{1cm}C{1cm}C{1cm}C{1cm}}
    \noalign{\vspace{2mm}}
    & $T^B$ & $T^C$ & $T^D$ & $T^E$ & $T^F$ & $T^{\DE}$ \\
    \noalign{\vspace{1mm}}
    \hline
    \noalign{\vspace{2mm}}
    ${\rm eff}_{T^{A},T^B}^{\ba}$ & 1 & $\frac{4}{9}$ & $\frac{2}{3}$ & $\frac{2}{3}$ & 0.8906 & 0.7618 \\
    \noalign{\vspace{2mm}}
    ${\rm eff}_{T^{A},T^C}^{\ba}$ &$\frac{1}{16}$ & 1 & $\frac{1}{4}$ & $\frac{1}{4}$ &  0.2818 & 0.2857 
\end{tabular}  
\vspace{3mm}

\caption{Local exact Bahadur efficiencies of the $T^A$ test, $A\in \{B,C,D,E,F,\DE\}$, with respect 
	to the BDY test $T^B$ and direction $q^{\rm FGM}$ (first line), and with respect to $T^C$ and 
	direction $q ^C$ (second line)}
\label{Tab:EffFGM}
\end{table}  
The local exact Bahadur efficiencies shown in the first line of Table \ref{Tab:EffFGM}
for the FGM family also apply to the copula families of Ali--Mikhail--Haq, Plackett, and Frank, 
given respectively by 
\begin{equation*}
  C_\theta^{\rm AMH}(x,y)\,=\,\frac{xy}{1-\theta (1-x)(1-y)},\quad -1\le \theta\le 1, 
\end{equation*}
\begin{equation*}
  C_\theta^{\rm P}(x,y)\,=\, \frac{1}{2\theta}\biggl(1+\theta (x+y) -
      \Bigl(\bigl (1+\theta (x+y)\bigr)^2 - 4xy\theta (\theta + 1)\Bigr)^{\nf{1}{2}}\biggr),
\end{equation*}
if $\theta \in (-1,\infty)\setminus \{0\}$, and $ C_0^{\rm P}(x,y)=xy$, 
\begin{equation*}
  C_\theta^{\rm F}(x,y)\,=\,
            \begin{cases} -\frac{1}{\theta} \log\left (1 +\frac{(e^{-\theta x}-1)(e^{-\theta y} -    
            		                                                        1)}{e^{-\theta}-1}\right ), &\text{if }\,\theta\neq 0.\\
                                        \quad xy,&\text{if }\,\theta = 0.
            \end{cases}
\end{equation*}
In fact, expanding $C_\theta^{\rm AMH}$, $C_\theta^{\rm P}$, and $C_\theta^{\rm F}$ in a Taylor series in powers of
$\theta$, the sum of the first two terms for $C_\theta^{\rm AMH}$ and $C_\theta^{\rm P}$ is
\begin{equation*}
  xy + \theta xy(1-x)(1-y),
\end{equation*}
and the sum of the first two terms for $C_\theta^{\rm F}$ is
\begin{equation*}
  xy + \frac{\theta}{2} xy(1-x)(1-y),
\end{equation*}
see \cite[pp.\,100,\,133]{Nelson}.

Other copula families also lead to factorizing alternatives. For example, for the generalized
Farlie--Gumbel--Morgenstern (GFGM) family
\begin{equation*}
  C_\theta(u,v)=uv\bigl (1+\theta(1-u^2)(1-v^2)\bigr )^2,\quad 0\le u,v\le 1, 
\end{equation*}
with $-\frac{1}{8}\le \theta \le \frac{1}{4}$, see \cite{BekParZad}, we get the direction
$q (u,w)=(1-3u^2)(1-3v^2)$, $0\le u,v\le 1$.
Moreover, the above factorizations may be expressed as the direction being a tensor 
of rank~1 and, at least in principle, calculations as in Example~\ref{ex:FGM} can be extended 
to directions that are tensors of finite degree. 

The theorem can also be used to show that none of the tests uniformly dominates the others. 
In fact, it is easy to see that $\kappa(K(A),q )\le 1$ and that equality holds if and only if $q $ is an 
element of the eigenspace associated with the largest eigenvalue of $K(A)$. As the latter 
are all different for the six $T^A$ tests the direction $q $ can always be chosen such that 
the ratio in~\eqref{eq:locexBah} is strictly less than 1. In Example~\ref{ex:FGM} the BDY test 
dominates all others, but with the direction 
\begin{equation}\label{eq:optCopC}
	 q ^C(u,v) := 2\cos(2\pi u)\cos(2\pi v),\quad 0\le u,v \le 1,
\end{equation}
i.e. an eigenfunction associated with the largest eigenvalue of $K(C)$,
we get the local exact Bahadur efficiencies in the second line of Table~\ref{Tab:EffFGM}:
For this direction the test $T^C$ dominates the others; see also Section \ref{sec:simul} for 
related simulation results. Likewise, for the direction 
\begin{equation*}
	 q ^B(u,v) := 2\cos(\pi u)\cos(\pi v),\quad 0\le u,v \le 1,
\end{equation*}
the eigenfunction associated with the largest eigenvalue of $K(B)$, the test BDY test
dominates the others.
Figure~\ref{fig:dir} shows the functions $q^B$ and $q ^C$ as gray plots on the unit square, 
with white and black corresponding to the smallest and largest values respectively. This may 
provide an idea of the respective alternatives that $T^B$ and $T^A$ are most sensitive to.  

\begin{figure}
\begin{center}
\setlength{\abovecaptionskip}{-0.25cm}
\includegraphics[scale=.42]{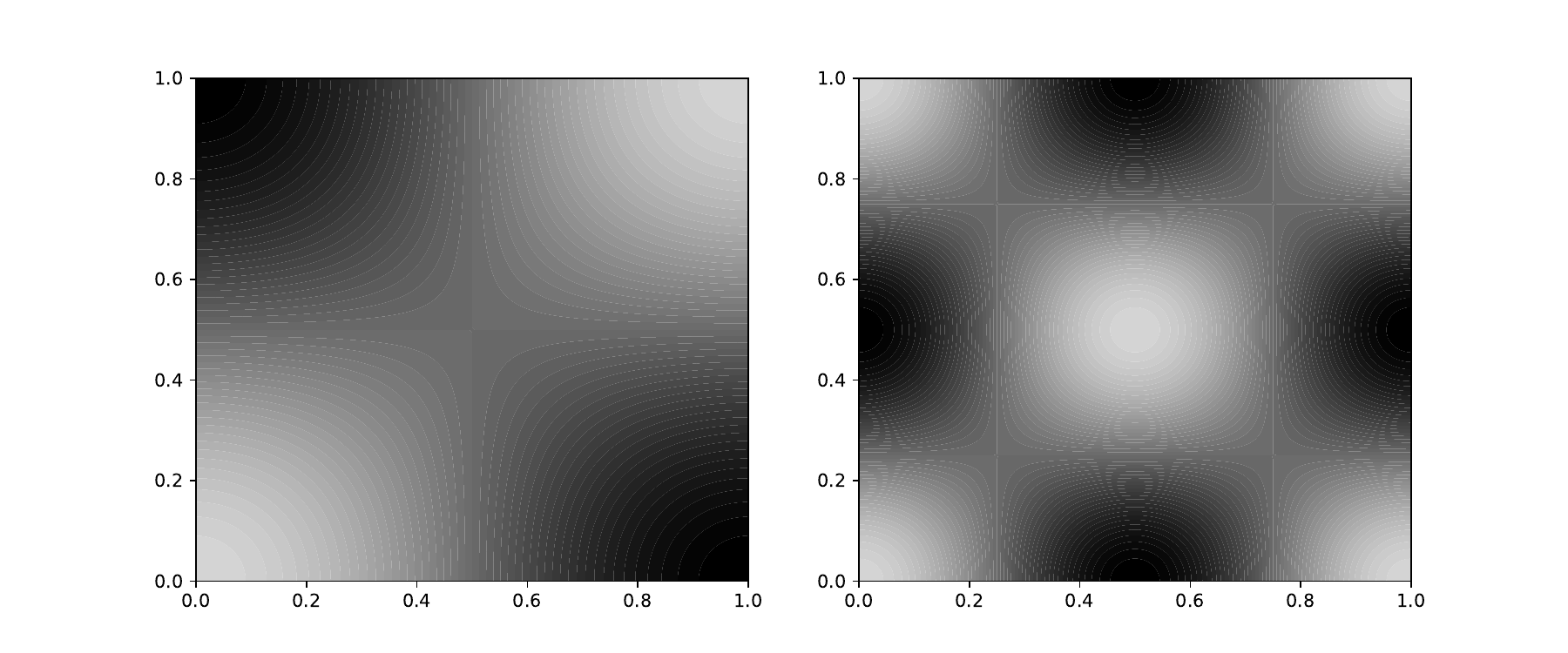}
\vspace{2mm}
\caption{Gray plots of $q^B$ (left) and $q^C$ (right)}
\label{fig:dir}
\end{center}
\end{figure}

\subsection{Other  efficiency concepts}\label{subsec:alteff} 
When dealing with Bahadur efficiencies the most difficult task often is the derivation 
of a suitable large deviation result for the test statistics under consideration.
To circumvent this problem, Bahadur \cite{Bahadur} introduced and discussed a related concept of approximate efficiency, where the large deviation statement is replaced by a condition on the exponential tail
behavior of the limiting null distributions. The \emph{approximate Bahadur efficiencies} obtained are 
ratios of the \emph{approximate slopes} of the respective sequences of test statistics; in typical cases, 
these slopes are calculated easily.   For a precise description of this approach we refer to \cite{Bahadur} 
and \cite{GroeneOoster}. Despite its disadvantages, see~\cite{GroeneOoster}, the concept
is still widely used as its local variant, the \emph{local approximate Bahadur efficiency}
${\rm eff}_{T_2,T_1}^{\ba^*}$ of a sequence $T_2$ with respect to a sequence $T_1$, often coincides with the
local exact Bahadur efficiency. This is easily seen to be true for the $T^A$ tests, where the
approximate slopes are given by $b_A(\theta)/\lambda_{A,1}$. In fact, using \eqref{eq:Ld:TA:H1} it follows that 
${\rm eff}_{T^{A'},T^A}^{\ba^*}:=\lim_{\theta\to 0}\frac{b_{A'}(\theta)/\lambda_{A',1}}{b_A(\theta)/\lambda_{A,1}}\,=\, 
\frac{\kappa(K(A'),q )}{\kappa(K(A),q )}={\rm eff}_{T^{A'},T^A}^{\ba}$.
It is also well-known that local approximate Bahadur efficiencies often coincide with a local variant of
Pitman efficiencies. Recall the situation at the beginning of Section \ref{subsec:baha}.
For given $0<\alpha<\beta<1$ the \emph{Pitman asymptotic relative efficiency} (ARE) of $T_2$ with respect to $T_1$
is defined to be 
${\rm eff}_{T_2,T_1}^{\pit}(\alpha,\beta):=\lim_{\theta\to \theta_0}\, \frac{N_{T_1}(\alpha,\beta,\theta)}{N_{T_2}(\alpha,\beta,\theta)}$,
if the limit exists; if the latter is not true, proceed analogously to the Bahadur efficiency approach. Call
${\rm eff}_{T_2,T_1}^{\pit}:=\lim_{\alpha \to 0}\, \frac{N_{T_1}(\alpha,\beta,\theta)}{N_{T_2}(\alpha,\beta,\theta)}$
the \emph{limiting Pitman efficiency} of $T_2$ with respect to $T_1$, if the limit exists. 
By a general theorem of Wieand \cite{Wieand} this is true if subject to a specific condition,
now called Wieand's condition ${\rm III}^*$, the local approximate Bahadur efficiency exists. Moreover, then
${\rm eff}_{T_2,T_1}^{\pit}={\rm eff}_{T_2,T_1}^{\ba^*}$; see \cite{Hoermann} and \cite{KallKon} for more recent work
and various examples. For sequences of quadratic tests, i.e.
tests based on quadratic $U$-statistics with degenerate kernel of rank 2 in the hypothesis case,
Gregory \cite[Theorem~3.1]{Gregory80} deals with properties of the
respective kernels ensuring that Wieand's condition ${\rm III}^*$ is satisfied. The $T^A$ are not quadratic tests,
but we may argue similarly, using Hoeffding's inequality for $U$-statistics with bounded kernels, see e.g.
\cite[Section\,5.6, Theorem A]{Serfling}. 
Finally, we also deal with an efficiency concept proposed by
Gregory~\cite{Gregory80} that is directly based on the power performance 
of the tests under special local alternatives; for this, the sample sizes $N_{T_i}(\alpha,\beta,\theta)$ are not needed.

We use the notation and the assumptions given above in~\eqref{eq:Ustat} - \eqref{eq:Ustatconv}, with
$P=P_{\theta_0}$ the distribution that specifies the null hypothesis 
$H_0$ of the testing problem considered at the beginning of Section \ref{subsec:baha}.
We concentrate on the comparison of asymptotically quadratic tests:
Given a sequence $T = (T_{n})_{n\in \bN}$ of tests (test statistics) we assume that there is a
sequence $U_{h_2} =(U_{h_2,n})_{n\ge 2}$ of
quadratic $U$-statistics $U_{h_2,n}=U_{h_2,n}(Z_1,\ldots,Z_n)$, with degenerate kernel $h=h_2$ of
rank 2 in the hypothesis case, such that, as $n\to\infty$,
\begin{equation}\label{eq:equiv:T:qua}
 n T_{n}(Z_1,\ldots,Z_n)\,-\,n U_{h_2,n}(Z_1,\ldots,Z_n)\,\to\,0\quad\text{in}\ P\text{-probability}.  
\end{equation}
These conditions are satisfied for the tests $T^A$, $A\in\{B,C,D,E,F,\DE\}$. 

Let $(P_n)_{n\in\bN}$ be a sequence of alternative distributions; we now assume that the sequence is local in the sense that
$P_n$ has a $P$-density $p_n$ that can be written as
\begin{equation}\label{eq:localt2}
    p_n = 1 + n^{-1/2}q_n,\,\text{ with }\, q_n\to q  \text{ for some  }q \in L^2=L^2(E,\cF,P),\, q\neq 0. 
  \end{equation}
Again, let $(\lambda_i,\phi_i)_{i\ge 1}$ be the spectral structure of the 
Hilbert--Schmidt operator $K$ associated with $h_2$. Gregory \cite{Gregory} showed that, 
under these assumptions and in extension of~\eqref{eq:Ustatconv},
\begin{equation*}
     \cL \left (nU_{h_2,n}(Z_{1},\ldots,Z_{n})\big|P_n\right )\toweak \mu(K,q ),
\end{equation*}
where the distribution $\mu(K,q )$ now has the representation 
\begin{equation}\label{eq:repr_alt0}
  \mu(K,q )\,=\,\cL\Bigl(\sum_{j \ge 1} \lambda_j \bigl((\xi_j+a_j)^2-1\bigr)\Bigr).
\end{equation} 
In~\eqref{eq:repr_alt0} the shifts are given by 
$a_j=\langle \varphi_j,q \rangle = \int \varphi_j q \,dP$, $j\in\bN$,
and the $\xi$-variables are again independent standard normals.
From~\eqref{eq:equiv:T:qua} and the fact that the sequence
$\bigl (P_n^n\bigr)_{n\in\bN}$ of $n$-fold product measures of $P_n$ is contiguous
to the sequence $(P^n)_{n\in\bN}$ of $n$-fold product measures of $P$ it follows that
\begin{equation*}
     \cL \left (nT_n(Z_{1},\ldots,Z_{n})\big|P_n\right )\toweak \mu(K,q ).
\end{equation*}
Extending the definition of $\kappa(K(A),q )$ given for $K(A)$, $A\in \{B,C,D,E,F,\DE\}$
in~\eqref{eq:defkappa}, let $\kappa(K,q ):=\frac{\langle Kq,q\rangle }{\lambda_1}$, with $\lambda_1$ the
largest eigenvalue of $K$.
Let $T^\circ=(T^\circ_n)_{n\in\bN}$ be some other sequence  of asymptotically quadratic
tests (test statistics) with Hilbert--Schmidt operator $K^\circ$. Then, following Gregory \cite{Gregory80} the ratio
\begin{equation*}
  {\rm eff}_{T_2,T_1}^{\gr}\,:=\,\kappa(T^\circ,q )/\kappa(T,q )
\end{equation*}
may be chosen as an alternative efficiency measure. 
Connecting the conditions~\eqref{eq:localt1} and~\eqref{eq:localt2} via
$q_n=q =q_\theta$ for all $n\in\bN$ and $\theta=\theta_n= n^{-1/2}$
we obtain in our situation that
\begin{equation}\label{eq:effsequal}
  {\rm eff}_{T^{A'},T^A}^{\ba}={\rm eff}_{T^{A'},T^A}^{\ba^*}={\rm eff}_{T^{'},T^A}^{\pit}={\rm eff}_{T^{A'},T^A}^{\gr}
\end{equation}
for $A,A'\in \{B,C,D,E,F,\DE\}$.

The quadratic statistics can also be related to Gaussian processes; see Section~\ref{subsec:processes}. 
With the Kac-Siegert representation~\eqref{eq:KacSiegert} of $X$ and 
\begin{equation*}
   R:\bH\to\bH,\quad 
     R f = \sum_{j\ge 1} \sqrt{\lambda_j}\,\langle  f,\varphi_j\rangle\, \varphi_j,
\end{equation*}
the `square root' of $K$ we obtain
\begin{equation}\label{eq:shift}
 \sum_{j\ge 1} \lambda_j(\xi_j+a_j)^2  \, \eqdistr\,  \|X+Rq\|^2.
\end{equation}
Note that we may then rewrite~\eqref{eq:defkappa} as $\kappa(K,q)=\|Rq\|^2/\lambda_1$.
This approach makes it possible to obtain, for a specific test,  
directions for the alternatives with the property that the limiting power is maximal. 
As shown in~\cite[Theorem 2.2]{Neuhaus} this holds with $q$ a multiple of any eigenfunction
$\varphi_1$ associated with the largest eigenvalue $\lambda_1$ of~$K$ (or, equivalently,
$\sqrt{\lambda_1}$ of $R$), which agrees with our findings at the end of Section~\ref{subsec:baha}

Gaussian processes can further be used in connection with local power functions of the tests; 
see~\cite[Section 2]{Neuhaus} for details and proofs. Instead  of~\eqref{eq:localt1}
or~\eqref{eq:localt2} we now consider local parametric families 
$\cP_n=\{P_{n,\theta}:\, \theta\in\Theta\}$, $n\in\bN$,  where $P_{n,\theta}$ denotes the 
distribution that has density $1+n^{-1/2}\theta q$ with respect to the uniform distribution 
distribution $P=P_0$ on the unit square. Here $q\in\bH_2$ is a fixed direction, 
$\Theta\subset \bR$ is an interval with $0$ as an interior point and 
such that $P_{n,\theta}$ defines a copula for all $n\in\bN$ 
and $\theta\in\Theta$.  We may in fact assume, reparametrizing if necessary, that $\|q\|=1$.
For the moment we ignore the dependence on $T^A$ and $q$ and write $\beta_n$ for the
power function of the $T_n^A$ test, $A\in\{B,C,D,E,F,\DE\}$.

Let $X$ be the centered Gaussian process given by the Kac--Siegert representation associated with
the operator $K(A)$. Regard $X$ as an $\bH_2$-valued random variable defined on some
background probability space with distribution
$\nu_0:=\cL(X)$, and let $c(\alpha)$ be the upper $\alpha$-quantile of $\|X\|^2$.  
It follows from~\eqref{eq:shift} that
\begin{equation}\label{eq:aspower}
 \lim_{n\to\infty}\beta_n(\theta)\,
               = \, \beta(\theta):= P(\|X+\theta Rq\|^2\ge c(\alpha)) 
               = \nu_\theta\bigl(\{h\in\bH_2:\, \|h\|^2\ge c(\alpha)\}\bigr),
\end{equation}
where $\nu_\theta$ denotes the distribution of $X+\theta Rq$.
The linear mapping 
\begin{equation*}
    Z_q:\bH_2\to \bR, \quad h\, \mapsto \, \sum_{j\ge 1} \frac{1}{\lambda_j} 
                          \langle Rq,\varphi_j\rangle \langle h,\varphi_j\rangle,
\end{equation*}
defines a random variable on the probability space $(\bH_2,\cB(\bH_2),\nu_0)$.
Linearity and $\cL(X)=\nu_0$ together imply that $Z_q$  has a 
centered normal distribution, and it is straightforward to check that its
variance is  $\|q\|^2=1$. Further, the shift of the Gaussian process $X$ by the deterministic
amount $\theta Rq$ leads to a formula for the respective densities,
 \begin{equation*}
 \frac{d\nu_\theta}{d\nu_0}(h) = \exp\bigl(\theta Z_q(h) - \theta^2/2\bigr),\quad h\in\bH_2.
\end{equation*}  
Taken together this shows that $\bigl(\bH_2,\cB(\bH_2),\{\nu_\theta:\, \theta\in\bR\}\bigr)$
is a Gaussian shift experiment, and~\eqref{eq:aspower} now leads to 
\begin{equation*}
  \beta(\theta) \, = \, \int_{\|h\|^2\ge c(\alpha)}
                            \exp\bigl(\theta Z_q(h)-\theta^2/2\bigr)\, \nu_0(dh).
\end{equation*} 
We have $\beta(-\theta)=\beta(\theta)$ and $\beta(0)=\alpha$, $\beta(\theta)>\alpha$ for
$\theta\not=0$. A quantitative measure of the performance of the test may thus be based on 
the curvature at the hypothesis, i.e.\ the second derivative 
of its power function at $\theta=0$, which turns out to be \cite[Lemma 2.5]{Neuhaus}
\begin{equation*}
  \frac{d^2}{d\theta^2}\beta(\theta)\big|_{\theta=0}
      \, = \, \int_{\| h \|^2\ge c(\alpha)} Z_q(h)^2\, \nu_0(d h)  \, -\, \alpha.
\end{equation*}  
Note that, in contrast to the previous efficiency assessments, this depends on the level $\alpha$.
 
These curvatures can be used to compare the different $T^A$ tests, 
for different directions $q$ and a fixed level $\alpha$, but we will not pursue this here. 
An  interesting alternative to simulation is the use of integral transforms, which is also 
applicable in connection with the limiting null distributions in Remark~\ref{rem:repr2};  
see~\cite{GenQueRem} for a related approach in the context of the Hoeffding--Blum--Kiefer--Rosenblatt test of independence. 
\section{Implementation aspects and simulations}\label{sec:simul}
For computationally intensive procedures a Monte Carlo approach is the somewhat canonical choice,
adding an approximation error (that is under control of the statistician) to the inherent stochastic 
error.  This is indeed suggested in~\cite{BergsmaDassios} and it seems to be the only reasonable 
method for longer patterns; see also~\cite[Section 6.1]{BaGr}.  For patterns of length four, as in 
the present case, we may avoid the approximation error by going through all subsets with four elements. 
This would require a run-time of order $n^4$ and is thus feasible only for very small values $n$ 
of the sample size. 
For our simulations below we used an ad hoc algorithm that reduces this to $n^3$. 
For an in-depth treatment of the computational complexity
we refer to the remarkable paper~\cite{EZLeng} where the authors introduce a data structure that leads
to an algorithm of rough (up to logarithmic factors) order $n^{3/2}$. In many cases,  
including the BDY statistic, it is even of roughly linear order.   

In order to obtain an impression of the power performance of the $T^A$ tests at a given 
level and for a finite sample size we present the result of
a simulation study.  

For sample size $n\in \{50,100\}$ and significance level $\alpha = 0.05$
we obtained the empirical power values of the tests by simulation with $10000$ replications,
and for alternative distributions taken from four different
copula families: (a) the FGM copula family \eqref{eq:defFGM}; (b) the Clayton copula family,
\begin{equation*}
  C_\kappa^{\rm Clayton}(u,v)\,
                  =\,\left (\max\left (u^{-\kappa}+v^{-\kappa} - 1,0\right )\right )^{-1/\kappa},
                    \ \ 0\le u,v\le 1,
\end{equation*}
with dependence parameter $\kappa\in [-1,\infty) \setminus \{0\}$; (c) the Gaussian copula family,
\begin{equation*}
  C_\rho^{\rm Gauss}(u,v)\,
                =\,\frac{1}{2\pi (1-\rho^2)}\int_{-\infty}^{\Phi^{-1}(u)}\int_{-\infty}^{\Phi^{-1}(v)}
                            \exp\Bigl (-\frac{s^2-2\rho st + t^2}{2(1-\rho^2)}\Bigr )\,ds \,dt, 
\end{equation*}
$0\le u,v\le 1$, with dependence parameter $\rho\in (-1,1)$ and $\Phi$ the 
distribution function of the standard normal distribution; and (d) the family 
\begin{equation*}
	C^{\rm opt,C}_\theta(u,v) \,=\,uv\,+\,\theta \frac{1}{2\pi^2}\sin(2\pi u)\sin(2\pi v),\quad 0\le u,v \le 1,
\end{equation*}
with dependence parameter $\theta\in [-1/2,1/2]$ arising from the direction $q^C$  given in~\eqref{eq:optCopC}.
The empirical upper $\alpha$ quantiles of the test statistics, obtained by
simulation with 100,000 replications in the case where the hypothesis is true, were chosen as
critical values (and as approximations of the true values $t_{n,\alpha}^A$).  The empirical power 
values are shown in Table \ref{tab:power}. Roughly, these reflect the efficiency results obtained 
in Section~\ref{sec:eff}; see Table~\ref{Tab:EffFGM}.

\begin{acks}[Acknowledgments]
  We thank Editor Arnak Dalalyan and two anonymous reviewers for constructive comments and suggestions.
  We thank Jon Noel for bringing the reference~\cite{CDN} to our attention.
\end{acks}

\begin{table}[p] \setlength{\tabcolsep}{1.45mm}
	\caption{Empirical power values for FGM copulas (a), Clayton copulas (b), Gaussian copulas (c), \\ and
		the $C^{\rm opt,C}$ copulas (d)}\label{tab:power}
{\small
	\begin{tabular}{ccccccccccccccc}
		\noalign{\vspace{3mm}}
		& \multicolumn{6}{c}{$n=50$} && \multicolumn{6}{c}{$n=100$}\\
		\noalign{\vspace{2mm}}
		$\theta$ && $T^B$ &  $T^C$ &  $T^F$ & $T^D$ & $T^E$ &  $T^{\DE}$ &&  $T^B$ &  $T^C$ &  $T^F$ & $T^D$ & $T^E$ &  $T^{\DE}$\\
		\noalign{\vspace{1mm}}
		$\nf{1}{10}$  && 0.06 & 0.05 & 0.05 & 0.05 & 0.05 & 0.05 && 0.06 & 0.05 & 0.06 & 0.05 & 0.06 & 0.06 \\
		$\nf{1}{4}$   && 0.08 & 0.05 & 0.07 & 0.06 & 0.06 & 0.07 && 0.13 & 0.06 & 0.10 & 0.08 & 0.08 & 0.09 \\ 
		$\nf{1}{2}$   && 0.19 & 0.07 & 0.16 & 0.11 & 0.11 & 0.12 && 0.36 & 0.11 & 0.28 & 0.19 & 0.19 & 0.22 \\                        
        $\nf{3}{4}$   && 0.39 & 0.11 & 0.30 & 0.20 & 0.20 & 0.24 && 0.69 & 0.19 & 0.57 & 0.39 & 0.39 & 0.45 \\                        $\nf{9}{10}$  && 0.54 & 0.14 & 0.43 & 0.29 & 0.28 & 0.34 && 0.84 & 0.27 & 0.73 & 0.54 & 0.55 & 0.61
	\end{tabular}
	
	\vspace{2mm}
	\begin{center} (a) \end{center}
	
	\begin{tabular}{ccccccccccccccc}
		\noalign{\vspace{3mm}}
		&& \multicolumn{6}{c}{$n=50$} && \multicolumn{6}{c}{$n=100$}\\
		\noalign{\vspace{2mm}}
		$\kappa$ && $T^B$ &  $T^C$ &  $T^F$ & $T^D$ & $T^E$ &  $T^{\DE}$ & & $T^B$ &  $T^C$ &  $T^F$ & $T^D$ & $T^E$ &  $T^{\DE}$\\
		\noalign{\vspace{1mm}}
		$-\nf{1}{2}$   &&  0.94 & 0.63 & 0.90 & 0.69 & 0.69 & 0.79 && 1.00 & 0.96 & 1.00 & 0.97 & 0.97 & 0.99 \\
		$-\nf{1}{4}$   &&  0.29 & 0.09 & 0.21 & 0.14 & 0.14 & 0.16 && 0.55 & 0.15 & 0.43 & 0.27 & 0.27 & 0.32 \\ 
		$-\nf{1}{10}$     &&  0.08 & 0.05 & 0.07 & 0.06 & 0.06 & 0.06 && 0.11 & 0.06 & 0.09 & 0.07 & 0.07 & 0.08 \\
		$\nf{1}{10}$      &&  0.08 & 0.05 & 0.07 & 0.06 & 0.06 & 0.07 && 0.10 & 0.06 & 0.08 & 0.07 & 0.07 & 0.07 \\
		$\nf{1}{4}$       &&  0.19 & 0.07 & 0.14 & 0.10 & 0.10 & 0.12 && 0.35 & 0.10 & 0.27 & 0.17 & 0.18 & 0.21 \\
		$\nf{1}{2}$       &&  0.51 & 0.15 & 0.39 & 0.25 & 0.25 & 0.30 && 0.83 & 0.30 & 0.72 & 0.51 & 0.51 & 0.59 \\
		$\nf{3}{4}$       &&  0.78 & 0.30 & 0.67 & 0.46 & 0.46 & 0.54 && 0.98 & 0.61 & 0.94 & 0.81 & 0.81 & 0.87 \\
		$\nf{9}{10}$      &&  0.88 & 0.42 & 0.79 & 0.58 & 0.58 & 0.67 && 1.00 & 0.77 & 0.99 & 0.91 & 0.91 & 0.95
	\end{tabular}
	
	\vspace{2mm}
	\begin{center} (b) \end{center}

	\begin{tabular}{ccccccccccccccc}
		\noalign{\vspace{3mm}}
		&& \multicolumn{6}{c}{$n=50$} && \multicolumn{6}{c}{$n=100$}\\
		\noalign{\vspace{2mm}}
		$\rho$ && $T^B$ &  $T^C$ &  $T^F$ & $T^D$ & $T^E$ &  $T^{\DE}$ & & $T^B$ &  $T^C$ &  $T^F$ & $T^D$ & $T^E$ &  $T^{\DE}$\\
		\noalign{\vspace{1mm}}
		$\nf{1}{10}$      && 0.09 & 0.06 & 0.07 & 0.06 & 0.06 & 0.07  &&  0.14 & 0.05 & 0.10 & 0.07 & 0.08 & 0.08 \\
		$\nf{1}{4}$       && 0.35 & 0.09 & 0.25 & 0.14 & 0.15 & 0.18  &&  0.61 & 0.14 & 0.46 & 0.28 & 0.28 & 0.33 \\
		$\nf{1}{2}$       && 0.92 & 0.38 & 0.81 & 0.56 & 0.56 & 0.64  &&  1.00 & 0.73 & 0.99 & 0.90 & 0.90 & 0.94
	\end{tabular}
	
	\vspace{2mm}
	\begin{center} (c) \end{center}

	\begin{tabular}{ccccccccccccccc}
		\noalign{\vspace{3mm}}
		&& \multicolumn{6}{c}{$n=50$} && \multicolumn{6}{c}{$n=100$}\\
		\noalign{\vspace{2mm}}
		$\theta$ && $T^B$ &  $T^C$ &  $T^F$ & $T^D$ & $T^E$ &  $T^{\DE}$ & & $T^B$ &  $T^C$ &  $T^F$ & $T^D$ & $T^E$ &  $T^{\DE}$\\
		\noalign{\vspace{1mm}}
		
		$-\nf{1}{2}$   && 0.05 & 0.79 & 0.38 & 0.25 & 0.24 & 0.33 && 0.11 & 0.99 & 0.87 & 0.71 & 0.71 & 0.81 \\
		$-\nf{1}{4}$   && 0.04 & 0.21 & 0.09 & 0.08 & 0.08 & 0.09 && 0.05 & 0.44 & 0.18 & 0.13 & 0.13 & 0.16 \\
		$-\nf{1}{10}$  && 0.05 & 0.07 & 0.06 & 0.05 & 0.06 & 0.06 && 0.05 & 0.10 & 0.06 & 0.06 & 0.06 & 0.06 \\
		$\nf{1}{10}$  && 0.05 & 0.07 & 0.05 & 0.05 & 0.05 & 0.05  && 0.06 & 0.11 & 0.07 & 0.06 & 0.06 & 0.07 \\
		$\nf{1}{4}$   && 0.07 & 0.21 & 0.10 & 0.08 & 0.08 & 0.09  && 0.08 & 0.45 & 0.17 & 0.12 & 0.12 & 0.15\\
		$\nf{1}{2}$   && 0.12 & 0.80 & 0.38 & 0.25 & 0.25 & 0.34  && 0.17 & 0.99 & 0.86 & 0.71 & 0.71 & 0.82           
	\end{tabular}
	
	\vspace{2mm}
	\begin{center} (d) \end{center}
}
\end{table}

\bibliographystyle{alpha } 
{}

\end{document}